\documentclass[reqno, 12pt]{amsart}

\usepackage{amsmath}
\usepackage{amsthm}
\usepackage{amsfonts}
\usepackage{amssymb}
\usepackage{mathrsfs}
\usepackage{epsfig}
\usepackage{slashed}
\usepackage{mathtools}
\usepackage{enumerate}
\usepackage{fullpage}
\usepackage{xcolor}
\usepackage{comment}
\usepackage[hidelinks]{hyperref}
\usepackage{framed}

\usepackage{enumitem}

\newcommand{\mc}{\mathcal}

\renewcommand{\Re}{\mathrm{Re}\,}

\newcommand{\ran}{\mathrm{ran}\,}

\newcommand{\N}{\mathbb{N}}
\newcommand{\R}{\mathbb{R}}
\newcommand{\C}{\mathbb{C}}

\renewcommand{\S}{\mathbb{S}}
\newcommand*\closure[1]{\overline{#1}}

\DeclarePairedDelimiter\abs{\lvert}{\rvert}

\DeclareMathOperator{\vol}{vol}

\newcommand\restr[2]{{  \left.\kern-\nulldelimiterspace #1 \vphantom{\big|}  \right|_{#2} }}

\newcommand\norm[1]{\left\Vert#1\right\Vert}

\newtheorem{lemma}{Lemma}[section]
\newtheorem{theorem}[lemma]{Theorem}

\newtheorem{proposition}[lemma]{Proposition}
\newtheorem{remark}[lemma]{Remark}

\theoremstyle{definition}
\newtheorem{definition}[lemma]{Definition}
\newtheorem*{definition*}{Definition}

\newtheorem*{assumption*}{Assumption}

\numberwithin{equation}{section}

\title[]{Stable blowup for the harmonic map heat flow into perturbed spheres}

\author{Alexander Wittenstein}
\address{Karlsruhe Institute of Technology, Institute for Analysis,  Englerstra{\ss}e 2, 76131 Karlsruhe, Germany}
\email{alexander.wittenstein@kit.edu}

\thanks{The author was funded by the
Deutsche Forschungsgemeinschaft (DFG, German Research Foundation) - Project-ID 258734477 - SFB 1173.}

\begin{document}

\begin{abstract}
    We prove stable self-similar blowup for the harmonic map heat flow in dimensions 3 to 6 for target manifolds which are slightly perturbed versions of the round sphere. Starting from the known blowup solution for the sphere, we construct self-similar blowup solutions for this class of target manifolds and prove that these solutions are asymptotically nonlinear stable.
    
    \noindent Although the blowup solution for the sphere is not explicitly known in this case, we can still use perturbative methods similar to those developed by Donninger, Schörkhuber, and the author in \cite{DonSchWit25} who consider a related problem for the wave maps equation.
\end{abstract}

\maketitle

\section{Introduction}

A harmonic map $U : (M^m,g) \to (N^d,h)$ between Riemannian manifolds is a critical point (with respect to compactly supported variations) of the Dirichlet energy \footnote{Here and throughout we use the Einstein summation convention where Latin indices range from 1 to $m$ and Greek indices from 1 to $d$.}
\begin{align}\label{Dirichlet energy}
    E[U] = \frac{1}{2} \int_M g^{ij}(x) \, h_{\alpha \beta}(U(x)) \, \frac{\partial U^{\alpha}}{\partial x^i}(x) \frac{\partial U^{\beta}}{\partial x^j}(x) \, d\vol_g(x).
\end{align}

The associated Euler-Lagrange equations are given by
\begin{align*}
    \tau(U) := \Delta_g U^{\alpha} + g^{ij} \, \Gamma_{\beta \gamma}^{\alpha} \circ U \, \frac{\partial U^{\beta}}{\partial x^i} \frac{\partial U^{\gamma}}{\partial x^j} = 0, \quad \text{for } ~ \alpha = 1,...,d
\end{align*}

where $\Delta_g$ is the Laplace-Beltrami operator on $M$ and $\Gamma_{\beta \gamma}^{\alpha}$ are the Christoffel-symbols on $N$.

In their seminal work \cite{EelSam64} Eells and Sampson established the fact that when the target manifold has nonpositive sectional curvature, one can obtain harmonic maps by evolving an initial map $U_0$ through their negative $L^2-$gradient flow via the so-called harmonic map heat flow
\begin{align}\label{Harmonic map heat flow}
    \partial_t U = \tau(U).
\end{align}

Their result shows that under these curvature assumptions the flow exists for all time and converges to a harmonic map in the given homotopy class. In settings beyond these curvature assumptions it is by now well-known that the flow may develop singularities in finite time.

A natural first step towards understanding this phenomenon is to look at the scaling properties of the equation. If we consider maps $U : [0,\infty) \times \R^d \to (N^d,h)$, the harmonic map heat flow \eqref{Harmonic map heat flow} is invariant under the parabolic scaling
\begin{align}\label{parabolic scaling}
    U_{\lambda}(t,x) := U(t/\lambda^2, x/\lambda)
\end{align}

and under this transformation the Dirichlet energy \eqref{Dirichlet energy} scales according to
\begin{align*}
    E[U_{\lambda}](t) = \lambda^{d-2} E[U](t/\lambda^2).
\end{align*}

This separates the analysis into the energy-critical case $d = 2$ and the energy-supercritical-case $d \geq 3$. The first constructions of finite-time singularities in these two regimes were obtained by Chang, Ding and Ye \cite{ChaDinYe92} for $d = 2$ and by Coron and Ghidaglia \cite{CorGhi89} for $d \geq 3$, both for maps into the round sphere. In the supercritical case Struwe showed in \cite{Str88} that if a solution develops a finite-time singularity, then any nontrivial blowup limit (after suitable rescaling) is either a self-similar shrinking solution or a harmonic sphere. That both of these scenarios can actually occur was proven in for example \cite{Fan99} by Fan and \cite{Bie15} by Biernat. 

In our work we focus on the first type of singularity formation, namely the one described by self-similar solutions. Since the target manifolds, which we will consider later on, are warped product manifolds which can be seen as small perturbations of the round sphere, we first start with the sphere as the target manifold and have a look at the \textit{corotational} harmonic map heat flow  into it.

More precisely, we consider maps $U: [0,\infty) \times \R^d \to \S^d$. Using spherical coordinates $(r,\omega)$ on $\R^d$ and polar coordinates $(u,\Omega)$ on $\S^d$ the corotational ansatz implies that $U$ can be written as
\begin{align*}
    U(t,r,\omega) = (u(t,r), \omega)
\end{align*}

for a radial profile function $u : [0,\infty) \times [0,\infty) \to \R$.

By plugging this ansatz into \eqref{Harmonic map heat flow} the harmonic map heat flow transforms into the following $d-$dimensional semilinear heat equation for the radial profile $u$
\begin{align}\label{Equation for radial profile into the sphere}
        \left(\partial_t - \partial^2_r - \frac{d-1}{r} \partial_r \right) u(t,r) + \frac{d-1}{r^2}\sin(u(t,r)) \cos(u(t,r)) = 0.
\end{align}

Self-similar solutions to this equation are now invariant under the natural parabolic scaling from \eqref{parabolic scaling}, combined with the time translation invariance, so that a self-similar solution is of the form
\begin{align*}
    u(t,r) = f\left(\frac{r}{\sqrt{T-t}}\right), \quad T > 0
\end{align*}

for a self-similar profile function $f$ which satisfies the following second order ordinary differential equation
\begin{align*}
   0 = f''(\rho) + \left(\frac{d-1}{\rho}-\frac{\rho}{2}\right) f'(\rho) - \frac{d-1}{\rho^2} \sin(f(\rho)) \, \cos(f(\rho)).
\end{align*}

Fan proved in \cite{Fan99} that \eqref{Equation for radial profile into the sphere} admits infinitely many self-similar solutions in dimensions $d \in \{3,4,5,6\}$. For the study of singularity formation, however, it is important to not only show the existence of such blowup solutions, but also to understand whether their behavior remains under small perturbations of the initial data. In dimensions $d = 3$, Biernat and Donninger constructed a particular self-similar profile in \cite{BieDon18} whose asymptotic nonlinear stability could then be proven by them in \cite{BieDonSch17} together with Schörkhuber. Recently, the corresponding result was extended to dimensions $4 \leq d \leq 6$ in \cite{AngKisSch26}. These results show that if the initial data are chosen sufficiently close to the corresponding self-similar solution, then the resulting evolution still blows up in finite time and converges back to the same profile after passing to similarity variables. Thus these profiles are not just special examples of finite-time blowup, but rather describe a stable blowup mechanism for the harmonic map heat flow.

In this work we construct self-similar blowup solutions for target manifolds that are small perturbations of the round sphere and then prove that these blowup solutions are also asymptotically nonlinear stable against small corotational perturbations of the initial data.

To be able to precisely state our main results, we now first introduce the class of target manifolds which will be considered in this paper, namely \textit{perturbed spheres}.

\begin{definition}[\cite{DonSchWit25}, Definition 1.1]\label{Definition: Perturbed sphere}
Let $\alpha \in C^{\infty}(\R)$ be a non-trivial, real-valued, even $2\pi-$periodic function with $\alpha(0) = \alpha(\pi) = 0$. For $\abs{\varepsilon} \leq \varepsilon_0 := \left(2\norm{\alpha}_{\infty}\right)^{-1}$ we define the warped product manifold $(S^d_{\varepsilon}, h)$ by
\begin{align}\label{Perturbed sphere}
    S^d_{\varepsilon}  := (0,\pi) \times_{w_{\varepsilon}} \S^{d-1}
\end{align}

equipped with the warping function $w_{\varepsilon} : \R \to \R$,
\begin{align}\label{Warping function}
    w_{\varepsilon}(u) := \sin(u) (1 + \varepsilon \, \alpha(u)).
\end{align}

In coordinates $(u,\Omega) \in (0,\pi) \times \S^{d-1}$ the metric on $S^d_{\varepsilon}$ is then given by 
\begin{align*}
     h = du^2 + w_{\varepsilon}(u)^2 d\Omega^2
\end{align*}

with $d\Omega^2$ denoting the standard round metric on $\S^{d-1} \hookrightarrow \R^d$.
\end{definition}
   
The assumptions on $\alpha$ in Definition \ref{Definition: Perturbed sphere} ensure that $S^d_{\varepsilon}$ is a smooth, compact, $d-$dimensional Riemannian manifold. For more details about warped product manifolds we refer to \cite{Bes78}. In the unperturbed case $\varepsilon = 0$, the warping function becomes $w_0(u) = \sin(u)$ and the corresponding manifold is the round sphere $S^d$. 

We next introduce \textit{normal coordinates} on $S^d_{\varepsilon}$, as in \cite{ShaTah94}, by setting
\begin{align}\label{Normal coordinates}
	U^j := u \, \Omega^j, \quad \text{for} \quad j=1,\dots,d.
\end{align}

In these coordinates, $S^d_{\varepsilon}$ together with its north pole can be identified with the ball $B_{\pi}(0) \subset \R^d$. Thus the harmonic map heat flow can be written as the following initial value problem on $[0,\infty) \times \R^d$
\begin{align}\label{HMHF in normal coordinates}
\begin{cases}
  \partial_t U = \tau(U), \\
  U(0,\cdot) = U_0
\end{cases}
\end{align}

for initial data $U_0: \R^d \to \R^d$.

A corotational map $U: [0,\infty) \times \R^d \to S_{\varepsilon}^d$ in these coordinates is of the form
\begin{align}\label{Corotational map}
    U(t,x) = u(t,\abs{x}) \frac{x}{\abs{x}}
\end{align}

and the harmonic map heat flow \eqref{Harmonic map heat flow} transforms into the following $d-$dimensional semilinear heat equation for the radial profile $u$
\begin{align}\label{Equation for radial profile}
    \left(\partial_t - \partial^2_r - \frac{d-1}{r} \partial_r \right) u(t,r) + \frac{d-1}{r^2}w_{\varepsilon}(u(t,r)) \, w_{\varepsilon}'(u(t,r)) = 0.
\end{align} 

As mentioned above, for $3 \leq d \leq 6$ and $\varepsilon = 0$, there exists an asymptotically nonlinear stable finite-time blowup solution $\widetilde{f}_0$ to this equation. More precisely
\begin{align}\label{Ground state}
    u_0^T(t,r) = \widetilde{f}_0\left(\frac{r}{\sqrt{T-t}}\right), \quad T > 0
\end{align}

defines a self-similar solution of \eqref{Equation for radial profile into the sphere}.

The first main result of this paper proves the existence of a self-similar finite-time blowup solution $\widetilde{f}_{\varepsilon}$ to the harmonic map heat flow into $S^d_{\varepsilon}$ for sufficiently small $\abs{\varepsilon} > 0$. 

\begin{theorem}\label{Theorem: Blowup Solution}
Let $3 \leq d \leq 6$. Then there exists an $\varepsilon^* > 0$ such that for every $\abs{\varepsilon} \leq \varepsilon^*$ and every $T > 0$ there exists a self-similar solution to equation \eqref{Equation for radial profile} whose radial derivative blows up in $r=0$ at time $T$. More precisely, there exists a $u_{\varepsilon}^T \in C^{\infty}([0,T) \times [0,\infty)) \cap L^{\infty}([0,T) \times [0,\infty))$ of the form
\begin{align*}
    u_{\varepsilon}^T(t,r) = \widetilde{f}_{\varepsilon}\left(\frac{r}{\sqrt{T-t}}\right) \quad \text{with } ~ \lim_{t \to T^{-}} \abs{\partial_r u_{\varepsilon}^T(t,0)} = \infty.
\end{align*}

Additionally, the profile $\widetilde{f}_{\varepsilon}$ can be written as 
\begin{align*}
    \widetilde{f}_{\varepsilon}(\rho) = \widetilde{f}_0(\rho) + \rho \, \widetilde{\phi}_{\varepsilon}(\rho)
\end{align*}

with $\widetilde{f}_0$ being the self-similar solution for $\varepsilon = 0$ defined in Eq.~\eqref{Ground state} and a perturbation $\widetilde{\phi}_{\varepsilon}$ which depends Lipschitz continuously on the parameter $\varepsilon$ in the sense that 
\begin{align*}
        \norm{\widetilde{\phi}_{\varepsilon} - \widetilde{\phi}_{\kappa}}_{W^{2,\infty}([0,\infty))} \lesssim \abs{\varepsilon-\kappa}
\end{align*}

holds for all $\abs{\varepsilon},\abs{\kappa} \leq \varepsilon^*$.
Furthermore, $\widetilde{f}_{\varepsilon}$ is odd with $0 < \widetilde{f}_{\varepsilon}(\rho) < \pi$ for all $\rho > 0$ and $\lim\limits_{\rho\to\infty} \widetilde{f}_{\varepsilon}(\rho)$ exists. Additionally, for every $k \in \N$ there are constants $ C_{\varepsilon,k}  > 0$ such that 
\begin{align}\label{Decay of f_e}
    \abs{\widetilde{f}_{\varepsilon}^{(k)}(\rho)} \leq C_{\varepsilon,k} \, \langle \rho \rangle^{-2-k}
\end{align}

holds for all $\rho  \in [0,\infty)$.
\end{theorem}

Since these constructed blowup solutions differ only via an $\varepsilon$-small perturbation from the original blowup solution we are able to obtain rigorous stability results for these solutions under small corotational perturbations. We formulate the result in normal coordinates and set, without loss of generality, the blowup time to $T=1$.

\begin{theorem}\label{Theorem: Stability in normal coordinates}
Let $3 \leq d \leq 6$ and take $\varepsilon^* > 0$ as in Theorem \ref{Theorem: Blowup Solution}. We then define for every $\abs{\varepsilon} \leq \varepsilon^*$
\begin{align*}
    U^T_{\varepsilon}(t,x) :=  \widetilde{f}_{\varepsilon}\left(\frac{\abs{x}}{\sqrt{T-t}}\right)  \frac{x}{\abs{x}}
\end{align*}

and consider corotational initial data of the form
\begin{align*}
U_0 = U^1_{\varepsilon}(0,\cdot) + \nu_0, 
\end{align*}

where $\nu_0: \R^d \to \R^d$ is defined as $\nu_0(x) = x v_0(\abs{x})$ for a radial Schwartz function $v_0(\abs{\cdot})  \in \mc{S}(\R^d)$. Let $(s,k) \in \R \times \N$  satisfy 
\begin{align}\label{Condition}
    \frac{d}{2} < s \leq \frac{d}{2} + \frac{1}{2d+2}, \quad k = d+3.
\end{align} 

Then there exists a strictly positive $\overline{\varepsilon} \leq \varepsilon^*$ such that for every $\abs{\varepsilon} \leq \overline{\varepsilon}$ there are constants $\delta > 0$ and $M_0 > 1$ such that for every $\nu_0(x) = x v_0(\abs{\cdot})$ as above with
\begin{align*}
    \norm{\nu_0}_{\dot{H}^s\cap\dot{H}^k(\R^d,\R^d)}< \frac{\delta}{M_0},
\end{align*}

there exists a $T = T_{\varepsilon} \in [1-\delta , 1+\delta]$ and a unique corotational map $U \in C^{\infty}([0,T)\times\R^d,\R^d)$ that satisfies \eqref{HMHF in normal coordinates} for all $(t,x) \in [0,T) \times \R^d$. The gradient of $U$ blows up at the origin as $t \to T^{-}$ and we have the decomposition 
\begin{align*}
    U(t,x) = U^{T}_{\varepsilon}(t,x) + \nu \left (t, \frac{x}{\sqrt{T-t}} \right ),
\end{align*}

for a function $\nu :[0,T) \times \R^d \to \R^d$ which satisfies 
\begin{align}\label{Decay of nu}
    \norm{\nu(t, \cdot)}_{\dot{H}^r(\R^d,\R^d)} \to 0 
\end{align}

as $t \to T^{-}$ for all $r \in [s,k]$. In addition,
\begin{align*}
    U(t,(\sqrt{T-t})x) \to \widetilde{f}_{\varepsilon}(\abs{x}) \frac{x}{\abs{x}}
\end{align*}

pointwise and uniformly on compact subsets of $\R^d$ as $t \to T^{-}$. 
\end{theorem}

\begin{remark}
Although the self-similar profile satisfies $0 < \widetilde{f}_{\varepsilon}(\rho) < \pi$ for all $\rho > 0$, the stability result does not assert that the radial profile function coming from the perturbed solution remains in the same interval. If the radial profile crosses $\pi$ the corresponding solution crosses the south pole of $S_{\varepsilon}^d$, which cannot be covered by a single uniform coordinate chart. Nevertheless, the corresponding solution can be smoothly continued via switching from the polar representation $(u,\omega)$ for $0 \leq u \leq \pi$ to $(2\pi - u, -\omega)$ for $\pi \leq u \leq 2\pi$, i.e. continuing the solution on the opposite meridian. 
\end{remark}

\subsection{Related results}
In this section we will just give a brief overview of self-similar blowup for the energy-supercritical harmonic map heat flow, for a more general discussion we refer the reader to for example \cite{LinWan08}. 

As we have already mentioned, the first self-similar blowup constructions are due to Coron and Ghidaglia \cite{CorGhi89} and Fan \cite{Fan99} in dimensions $3 \leq d \leq 6$. It was later shown by Bizoń and Wasserman \cite{BizWas15} that these are precisely the dimensions in which self-similar shrinking solutions can exist. We remark that if we allow for $k$-equivariant maps into spheres of higher dimension than the domain, Gastel showed in \cite{Gas02} that self-similar solutions exist for all supercritical dimensions $d \geq 3$.

The first stability result was obtained by Biernat and Donninger in \cite{BieDon18} and together with Schörkhuber in \cite{BieDonSch17} they constructed a spectrally stable self-similar profile in $d = 3$ dimensions and proved its asymptotic nonlinear stability. This result was now recently proven by \cite{AngKisSch26} in dimensions 4 to 6 as well. 

In dimensions $d \geq 4$ Glogi\'c, Kistner and Schörkhuber proved the existence and stability of an explicit self-similar blowup profile into a compact, rotationally symmetric target manifold. This result shows that stable self-similar blowup persists also in geometric settings outside of the round sphere. In our work, instead of relying on an explicit profile for a specifically chosen target manifold, we prove that stable self-similar blowup persists under small geometrical perturbations of the sphere. Thus, our result does not only provide another example of stable blowup, but also shows the stability of the blowup mechanism itself with respect to the sphere as the target manifold.


\subsection{Outline}

We start with the $d$-dimensional semilinear heat equation \eqref{Equation for radial profile} for the radial profile $u$. Setting $\widetilde{v}(t,r) := r^{-1} u(t,r)$ the equation transforms into a $(d+2)-$dimensional radial semilinear heat equation with a smooth nonlinearity
\begin{align}\label{Cauchy problem 2}
    \left(\partial_t - \partial^2_r - \frac{d+1}{r} \partial_r \right)\widetilde{v}(t,r) - \frac{d-1}{r^3}\left(r\,\widetilde{v}(t,r)-w_{\varepsilon}(r\,\widetilde{v}(t,r))\,w_{\varepsilon}'(r\,\widetilde{v}(t,r))\right) = 0.
\end{align}

For $n := d+2$ and $v(t,x) := \widetilde{v}(t,\abs{x})$ we can reformulate \eqref{Cauchy problem 2} as a heat equation on $\R^n$
\begin{align}\label{Nonlinear heat equation in n dimensions}
 	(\partial_t  - \Delta_x) v(t,x) = \frac{n-3}{\abs{x}^3}\left(\abs{x}v(t,x) - w_{\varepsilon}(\abs{x}v(t,x))w_{\varepsilon}'(\abs{x}v(t,x))\right), \quad x \in \R^n.
\end{align}

We then pass to similarity coordinates 
 \begin{align*}
 	\tau = \log\left(\frac{T}{T-t}\right), \quad y = \frac{x}{\sqrt{T-t}}
 \end{align*}
for $T > 0$ and $(t,x) \in [0,T) \times \R^n$ and in these coordinates \eqref{Nonlinear heat equation in n dimensions} turns into an evolution equation of the form
\begin{align}\label{Outline: Evolution equation}
    \partial_{\tau}\psi(\tau) = L\psi(\tau) + N_{\varepsilon}(\psi(\tau)), 
\end{align}
where $L$ generates the free heat evolution in similarity coordinates and $N_{\varepsilon}$ is the nonlinear part with the parameter $\varepsilon$ appearing through the warping function $w_{\varepsilon}$. A detailed derivation of this equation is given in Section \ref{Section: Existence of self-similar solutions} below. We study these operators in intersection homogeneous Sobolev spaces of radial functions
\begin{align*}
    X_s^k= \dot{H}_r^s(\R^n) \cap \dot{H}_r^k(\R^n)
\end{align*}

for suitable exponents $\frac{n}{2} - 1 < s < \frac{n}{2} < k$, $k \in \N$.

\subsubsection{Existence of self-similar blowup profiles}

Since self-similar solutions become static in similarity coordinates we are looking for a $\psi_{\varepsilon}$ solving
\begin{align}\label{Outline: Static solution}
    L \psi_{\varepsilon} + N_{\varepsilon} (\psi_{\varepsilon}) = 0.
\end{align} 
Due to the fact that the geometry of the perturbed spheres is somehow close to the geometry of the perfectly round $d$-sphere it is natural to look for self-similar solutions which are also somehow close to the original self-similar solution for the sphere. 
We therefore make the ansatz $\psi_{\varepsilon} = \psi_0 + \phi_{\varepsilon}$ where $\psi_0$ is a static solution to \eqref{Outline: Evolution equation} for $\varepsilon = 0$ and obtain a perturbation equation for $\phi_{\varepsilon}$ which can be written as
\begin{align}\label{Outline: Perturbation equation}
    - L_0 \phi_{\varepsilon} = V_{\varepsilon}(\psi_0)\phi_{\varepsilon}+ \widetilde{N}_{\varepsilon}(\phi_{\varepsilon}) + \mc{R}_{\varepsilon}(\psi_0).
\end{align}

The key idea here is that we do not want to invert the entire linear part of this equation but only $L_0:= L + V_0(\psi_0)$ which is precisely the linearization around $\psi_0$ in the case where $\varepsilon$ is equal to 0. The warping function $w_{\varepsilon}$ from \eqref{Warping function} is chosen in such a way so that the remaining potential $V_{\varepsilon}(\psi_0)$ can be made arbitrarily small (for sufficiently small $\varepsilon$). The same holds true for the remainder term $\mc{R}_{\varepsilon}(\psi_0)$ and the (now quadratically small) nonlinearity $\widetilde{N}_{\varepsilon}$. A precise definition of all these operators can be found in Section \ref{Section: Existence of self-similar solutions}.

The linearized operator $L_0$ is known to be self-adjoint in an exponentially weighted $L^2$-space and that its spectrum consists only of isolated eigenvalues with finite multiplicity. The same result holds true if we restrict the operator onto $X_s^k$ in which space we can then consider the right-hand side of \eqref{Outline: Perturbation equation}, see \cite{AngKisSch26}. To prove the necessary estimates for these operators we use the parameter-dependent Schauder type estimates from \cite{DonSchWit25}, Proposition A.1, where we will assume 
\begin{align}\label{Outline: Condition on exponents}
\frac{n}{2}-1 < s \leq \frac{n}{2}-1 + \frac{1}{2n-2}, \quad k = n + 1.
\end{align}

With that we can show that the operators on the right-hand side of \eqref{Outline: Perturbation equation} define maps $X_s^k \to X_s^k$ which are Lipschitz-continuous with respect to the parameter $\varepsilon$, see Lemma \ref{Lemma: Operator estimates}. 

Due to the invertibility of $L_0$ in $X_s^k$ we rewrite \eqref{Outline: Perturbation equation} as a fixed-point equation and apply Banach's fixed-point theorem for sufficiently small $\varepsilon > 0$.

The smoothness and decay properties of the resulting solution $\psi_{\varepsilon}$ to Eq.~\eqref{Outline: Static solution} follow from an ODE analysis similar to \cite{BieDon18}. 
Finally, transforming back to the original variables completes the proof of Theorem \ref{Theorem: Blowup Solution}.

\subsubsection{Stability of the self-similar blowup solution}

In Section \ref{Section: Stability Analysis} we show the asymptotic nonlinear stability of the blowup solution from Theorem \ref{Theorem: Blowup Solution}. We start with the semilinear heat equation \eqref{Nonlinear heat equation in n dimensions} on $\R^n$ with perturbed initial data around the self-similar blowup solution that blows up at time $T=1$. That is we are considering
\begin{align*}
    \begin{cases}
    \partial_t v - \Delta_x v = \frac{n-3}{\abs{x}^3}\left(\abs{x}v - w_{\varepsilon}(\abs{x}v)w_{\varepsilon}'(\abs{x}v)\right) \\
    v(0,x) = v_{\varepsilon}^1(0,x) + \varphi_0(x). 
    \end{cases}
\end{align*}

We then again introduce similarity coordinates to obtain
\begin{align*}
    \begin{cases}
        \partial_{\tau}\psi(\tau) = L \psi(\tau)+ N_{\varepsilon}(\psi(\tau)), \\
        \psi(0) = \psi_{\varepsilon}^T + \varphi_0^T
     \end{cases}
\end{align*}

and note that the only trace of the time parameter $T$ now lies in the initial condition. If we then make the (now time-dependent) ansatz $\psi_{\varepsilon}(\tau) = \psi_{\varepsilon} + \phi_{\varepsilon}(\tau)$, linearize around $\psi_{\varepsilon}$, we then end up with the central evolution equation of this paper
\begin{align}\label{Outline: Central evolution equation}
    \begin{cases}
        \partial_{\tau}\phi_{\varepsilon}(\tau) = L_{\varepsilon} \phi_{\varepsilon}(\tau) +  \widehat{N}_{\varepsilon}(\phi_{\varepsilon}(\tau)),\\
        \phi_{\varepsilon}(0) = U_{\varepsilon,T}(\varphi_0). 
    \end{cases}
\end{align}

Here $L_{\varepsilon}$ denotes the linearization around the self-similar profile $\psi_{\varepsilon}$. More precisely, it is given by the heat operator in similarity coordinates together with a potential term depending on $\psi_{\varepsilon}$. The remaining nonlinearity $\widehat{N}_{\varepsilon}$ is quadratically in its argument and $U_{\varepsilon,T}$ is the initial data operator and contains the dependence on the blowup time $T$. For the precise definition of the above defined operators we refer to Section \ref{Section: Stability Analysis}.

The stability problem is therefore reduced to the analysis of the Cauchy problem \eqref{Outline: Central evolution equation}. If one could construct global solutions which decay exponentially as $\tau$ goes to $\infty$ for all sufficiently small initial data, then the asymptotic stability of the profile $\psi_{\varepsilon}$ would follow.

However, the linearized operator $L_{\varepsilon}$ has a simple unstable eigenvalue at $\lambda = 1$, see Lemma \ref{Lemma: Eigenfunction}. Thus one cannot expect exponential decay for arbitrary small initial data without first removing this unstable direction. 

Therefore, the main task is to show that this is the only unstable eigenvalue of the linearized operator. For $\varepsilon = 0$ the corresponding spectral properties are known: apart from the simple eigenvalue $\lambda = 1$, the spectrum of the linearized operator $L_0$ is contained in a left half plane, see Proposition \ref{Proposition: Properties of L_0}. We prove that this spectral structure remains the same for $L_{\varepsilon}$ for $\varepsilon$ sufficiently small. The argument will follow the perturbative strategy used in \cite{DonSchWit25}, Section 4. More precisely, we first use a Neumann-series argument to show that the spectrum of $L_{\varepsilon}$ is contained (uniformly for sufficiently small $\varepsilon$) in a left half-plane together with a fixed compact region around the unstable eigenvalue $\lambda = 1$.

For the remaining compact region we use the fact that the Riesz-projections also depend (Lipschitz-)continuous on the parameter $\varepsilon$ so that there can not exist any other spectral points in that region due to dimensional reasons, see for example \cite{Kat95}, p.34, Lemma 4.10.

The artificial instability (generated by the time translation invariance with respect to the blowup time $T = 1$) can be dealt with the standard approach via adding a suitable correction term to the initial data operator.

Combining the above arguments, we can show the following stability result:

\begin{theorem}\label{Theorem: Stability of blowup solution}
Let $5 \leq n \leq 8$ and take $\varepsilon^* > 0$ as in Theorem \ref{Theorem: Blowup Solution}. For $\abs{\varepsilon} \leq \varepsilon^*$ we define
\begin{align*}
    v_{\varepsilon}^T(t,x) := \frac{1}{\sqrt{T-t}} \psi_{\varepsilon}\left(\frac{x}{\sqrt{T-t}}\right)
\end{align*}
and take $(s,k) \in \R \times \N$ satisfying
\begin{align}\label{Condition on exponents in n-dimensions}
    \frac{n}{2}-1<s \leq \frac{n}{2}-1+\frac{1}{2n-2}, \quad k = n + 1.
\end{align}

Then there exist $\omega>0$ and $0 < \overline{\varepsilon} \leq \varepsilon^*$ such that for every $\abs{\varepsilon} \leq \overline{\varepsilon}$ there are $\delta >0$ and  $M > 1$ such that the following holds: For every real-valued $\varphi_0 \in \mc{S}_r(\R^n)$ satisfying
\begin{align*}
    \norm{\varphi_0}_{\dot{H}^s\cap\dot{H}^k(\R^n)}< \frac{\delta}{M},
\end{align*}

there exists a $T = T_{\varepsilon} \in [1-\delta , 1+\delta]$ and a unique radial solution $v \in C^{\infty}([0,T)\times\R^n)$ to Eq.~\eqref{Nonlinear heat equation in n dimensions} with	
\begin{align*}
    v(0,\cdot) =v_{\varepsilon}^1(0,\cdot) + \varphi_0.
\end{align*}

Moreover, $v$ blows up at $(T,0)$ and can be decomposed as 
\begin{align*}
    v(t,x) = v_{\varepsilon}^{T}(t,x) + \frac{1}{\sqrt{T-t}} \varphi \left (\log\left(\frac{T}{T-t}\right), \frac{x}{\sqrt{T-t}} \right)
\end{align*}

for all $(t,x) \in [0,T) \times \R^n$, where $\varphi \in C^{\infty}([0,\infty)\times\R^n)$ is radially symmetric and satisfies
\begin{align*}
\norm{\varphi(-\log(T-t) + \log T,\cdot)}_{\dot{H}^r(\R^n)} \lesssim \delta (T-t)^{\omega}
\end{align*}
for all $r \in [s,k]$.
\end{theorem}

To then finally obtain Theorem \ref{Theorem: Stability in normal coordinates} we use the equivalence of norms of corotational maps and their radial profiles, see Proposition A.5 \cite{Glo22} .

\subsection{Notations}\label{Section: Notations} 

For $n \in \N$ we denote by $C^\infty_r(\R^n):=\{f\in C^\infty(\R^n): f\text{ is radial}\}$ the set of smooth radial functions and by $C^\infty_{c,r}(\R^n)$ the ones with additional compact support. By
\begin{align*}
    C_e^{\infty}[0,\infty) := \{f \in C^{\infty}([0,\infty)) : f^{(2j+1)}(0) = 0 \text{ for }j \in \mathbb N_0\},
\end{align*}

we define the set of smooth and even functions and note that there is a one-to-one correspondence between $C_r^{\infty}(\R^n)$ and $C_e^{\infty}[0,\infty)$. For this we will use the following convention: Is $\psi \in C_r^{\infty}(\R^n)$ a radial function we denote its associated radial profile via $\widetilde{\psi} \in C_e^{\infty}[0,\infty)$, i.e. $\psi(x) = \widetilde{\psi}(|x|)$.
Furthermore
\begin{align*}
    \mc{S}(\R^n) := \{f \in C^{\infty}(\R^n) : \forall \alpha, \beta \in \N_0^n : \sup\limits_{x \in \R^n} \abs{x^{\alpha} D^{\beta}f(x)} < \infty \}
\end{align*}
denotes the set of Schwartz functions and $\mc{S}_r(\R^n)$ denotes the subspace of radially symmetric Schwartz functions. 

Next we define the Fourier transform $\mc{F}u$ of $u \in C_c^{\infty}(\mathbb{R}^n)$ via
\begin{align*}
    \mc{F}u(\xi)=\frac{1}{(2\pi)^{\frac{n}{2}}} \int_{\mathbb{R}^n} u(x) e^{-i\xi\cdot x} \, dx.
\end{align*}

For $u,v \in C_c^{\infty}(\mathbb{R}^n)$ and $s \geq 0$ we define the inner product
\begin{align*}
    \langle u,v \rangle_{\dot{H}^s(\mathbb{R}^n)} := \langle \abs{\cdot}^s\mc{F}u, \abs{\cdot}^s\mc{F}v \rangle_{L^2(\mathbb{R}^n)}, 
\end{align*}

and the induced norm $\norm{u}_{\dot{H}^s(\mathbb{R}^n)}^2 := \langle u,u \rangle_{\dot{H}^s(\mathbb{R}^n)}$. If $k \in \N_0$ is a nonnegative integer
\begin{align*}
    \norm{u}_{\dot{H}^k(\mathbb{R}^n)} \simeq \sum_{\abs{\beta}=k} \norm{\partial^{\beta}u}_{L^2(\mathbb{R}^n)}
\end{align*}

for all $u\in C_c^{\infty}(\mathbb{R}^n)$. Finally, for $s \geq 0$ the homogeneous radial Sobolev space $\dot{H}_r^s(\R^n)$ denotes the space which is obtained by completion of radial test functions $C_{c,r}^{\infty}(\mathbb{R}^n)$ with respect to the above defined norm.

\section{Existence of self-similar solutions into \texorpdfstring{$S_{\varepsilon}^d$}{TEXT}}\label{Section: Existence of self-similar solutions}

From here on out we always assume $3 \leq d \leq 6$ and that $\alpha \in C^{\infty}(\R)$ satisfies the assumptions from Definition \ref{Definition: Perturbed sphere}. We then consider for $\varepsilon \in \R$ with $\abs{\varepsilon} \leq \varepsilon_0$ the family $S_{\varepsilon}^d$ of warped product manifolds with corresponding warping function $w_{\varepsilon}$ defined as in \eqref{Perturbed sphere} and \eqref{Warping function}, respectively. Our aim is now to construct for sufficiently small $\varepsilon$ a self-similar solution to the following $n := d+2-$dimensional semilinear heat equation
\begin{align}\label{Semilinear heat equation in n dimensions}
    \partial_t v - \Delta_x v = \frac{n-3}{\abs{x}^3} \left(\abs{x}v - w_{\varepsilon}(\abs{x}v) w_{\varepsilon}'(\abs{x}v)\right).
\end{align}

For this we first of all introduce similarity coordinates
\begin{align*}
    \tau = \log\left(\frac{T}{T-t}\right) \quad \text{and} \quad y = \frac{x}{\sqrt{T-t}}
\end{align*}

so that $t$ and $x$ can be written as
\begin{align*}
    t = T - T e^{-\tau} \quad \text{and} \quad x = \sqrt{T} e^{-\frac{\tau}{2}} y
\end{align*}

and observe that the differential operators transform in the following way
\begin{align*}
    \partial_t = \frac{e^{\tau}}{T}\left(\partial_{\tau} + \frac{1}{2} y \cdot \nabla_y\right) \quad \text{and} \quad \Delta_x = \frac{e^{\tau}}{T} \Delta_y.
\end{align*}

We now set
\begin{align*}
    \psi(\tau,y) := \sqrt{T} e^{-\frac{\tau}{2}} v(T-Te^{-\tau}, \sqrt{T}e^{-\frac{\tau}{2}}y) \left[= \sqrt{T-t} \, v(t,x)\right]
\end{align*}

so that \eqref{Semilinear heat equation in n dimensions} becomes
\begin{align}\label{Equation in similarity coordinates}
    \left(\partial_{\tau} - \Delta_y + \Lambda\right) \psi(\tau,y) = \frac{n-3}{\abs{y}^3} \left(\abs{y}\psi(\tau,y) - w_{\varepsilon}(\abs{y}\psi(\tau,y)) \,  w_{\varepsilon}'(\abs{y}\psi(\tau,y))\right),
\end{align}

where $\Lambda$ is defined as $\Lambda f(y) = \frac{1}{2}\left(y \cdot \nabla_y f(y) + f(y)\right)$. We now use the notation $\psi(\tau)(y) := \psi(\tau,y)$ and rewrite \eqref{Equation in similarity coordinates} as an evolution equation
\begin{align}\label{Evolution equation}
    \partial_{\tau} \psi(\tau) = \widetilde{L} \, \psi(\tau) + N_{\varepsilon}(\psi(\tau)),
\end{align}

where the linear operator is given by $\widetilde{L} := \Delta - \Lambda$ and the nonlinearity is defined as
\begin{align}\label{Nonlinearity}
    N_{\varepsilon}(\psi(\tau))(y) = \frac{n-3}{\abs{y}^3} \eta_{\varepsilon}(\abs{y}\psi(\tau,y)) \quad \text{for } ~ \eta_{\varepsilon}(z) := z - w_{\varepsilon}(z) \, w_{\varepsilon}'(z).
\end{align}

Since self-similar solutions become static in similarity coordinates we are looking for a $\psi_{\varepsilon}$ solving
\begin{align}\label{Static equation}
    0 = \widetilde{L} \, \psi_{\varepsilon} + N_{\varepsilon}(\psi_{\varepsilon}).
\end{align}

The idea is now to linearize about the self-similar blowup solution for the sphere by making the ansatz $\psi_{\varepsilon} = \psi_0 + \phi_{\varepsilon}$,
where $\psi_0$ is given by $\psi_0(y) = \abs{y}^{-1} \widetilde{f}_0(\abs{y})$ and $\widetilde{f}_0$ being the self-similar profile from \eqref{Ground state}.

Since we have $0 = \widetilde{L} \, \psi_0 + N_0(\psi_0)$ plugging this ansatz into \eqref{Static equation} leads to the following perturbation equation for $\phi_{\varepsilon}$
\begin{align}\label{Perturbation equation}
    0 = \widetilde{L}_0 \, \phi_{\varepsilon} + V_{\varepsilon}(\psi_0) \phi_{\varepsilon} + \widetilde{N}_{\varepsilon}(\phi_{\varepsilon}) + \mc{R}_{\varepsilon}(\psi_0),
\end{align}

where we have defined $\widetilde{L}_0 := \widetilde{L} + V_0(\psi_0)$ for
\begin{align*}
    V_0(\psi_0) = \frac{n-3}{\abs{y}^2} \eta_0'(\abs{y}\psi_0)  
\end{align*}

and the other operators are given by
\begin{align*}
    V_{\varepsilon}(\psi_0) & = \frac{n-3}{\abs{y}^2} \left(\eta_{\varepsilon}'(\abs{y}\psi_0) - \eta_0'(\abs{y}\psi_0)\right),\\
    \mc{R}_{\varepsilon}(\psi_0) & = \frac{n-3}{\abs{y}^3} \left(\eta_{\varepsilon}(\abs{y}\psi_0) - \eta_0(\abs{y}\psi_0)\right),\\ \widetilde{N}_{\varepsilon}(u) & = \frac{n-3}{\abs{y}^3} \left(\eta_{\varepsilon}(\abs{y}(\psi_0+u)) - \eta_{\varepsilon}(\abs{y}\psi_0) - \eta_{\varepsilon}'(\abs{y}\psi_0) \abs{y}u\right). 
\end{align*}

In the following we introduce the function spaces in which the blowup solution is constructed and in which the stability analysis is carried out later on.

\subsection{Function spaces} For $0 \leq s < k$ we define on $C_{c,r}^{\infty}(\R^n)$ an inner product via
\begin{align*}
    \langle u, v \rangle_{s,k} := \langle u, v \rangle_{X_s^k(\R^n)} := \langle u, v \rangle_{\dot{H}^s(\R^n)} + \langle u, v \rangle_{\dot{H}^k(\R^n)} 
\end{align*}

and $X_s^k := X_s^k(\R^n)$ is then defined as the completion of radial test functions $C_{c,r}^{\infty}(\R^n)$ with respect to the induced norm $\norm{\cdot}_{s,k}$. Additionally, we define 
\begin{align*}
    \mc{H} := \left\{f : \R^n \to \C ~ \text{measurable and radial with} ~ \int_{\R^n} \abs{f(x)}^2 e^{-\frac{\abs{x}^2}{4}} \, dx < \infty\right\}
\end{align*}

as the exponentially weighted $L^2-$space of radial functions.\medskip

To later on obtain Lipschitz estimates (with respect to $\varepsilon$) of the operators occurring in \eqref{Perturbation equation} we need parameter depending Schauder-type estimates. These were proven in \cite{DonSchWit25}, Proposition A.1, and we state them here again for convenience.

\begin{lemma}[Proposition A.1 in \cite{DonSchWit25}]\label{Lemma: Schauder}
Let $n \geq 5$ and $\abs{\varepsilon} \leq 1$ let $F_{\varepsilon} \in C^\infty(\R)$ be a family of even functions such that for all $\ell \in \mathbb{N}_0$ there exists a constant $C_{\ell} \geq 0$ such that
\begin{align} 
	\abs{F_{\varepsilon}^{(\ell)}(x) - F_{\kappa}^{(\ell)}(y)} \leq C_{\ell} \, \left( \abs{\varepsilon - \kappa} + \abs{x-y} \right)
\end{align}

holds for all $x,y \in \R$ and all $\abs{\varepsilon}, \abs{\kappa} \leq 1$. Then, for every $s\in\R$ and $k\in\N$ that satisfy
\begin{align}
    \frac{n}{2}-1 < s \leq \frac{n}{2}-1 + \frac{1}{2n-2}, \quad k>n
\end{align}

we have
\begin{align}
	\norm{u_1u_2u_3 \left(F_{\varepsilon}(\abs{\cdot}v)-F_{\kappa}(\abs{\cdot}v)\right)}_{\dot{H}^{s-1} \cap \dot{H}^{k}(\mathbb{R}^n)} \lesssim \abs{\varepsilon-\kappa} \prod_{i=1}^{3} \norm{u_i}_{\dot{H}^{s} \cap \dot{H}^{k}(\mathbb{R}^n)} \sum_{j=0}^{k}\norm{v}_{\dot{H}^{s} \cap \dot{H}^{k}(\mathbb{R}^n)}^{2j}
\end{align}

as well as
\begin{align}
\begin{split}
	&\norm{u_1u_2u_3 \left(F_{\varepsilon}(\abs{\cdot}v_1)-F_{\varepsilon}(\abs{\cdot}v_2)\right)}_{\dot{H}^{s-1} \cap \dot{H}^{k}(\mathbb{R}^n)}\\ &\lesssim \prod_{i=1}^{3} \norm{u_i}_{\dot{H}^{s} \cap \dot{H}^{k}(\mathbb{R}^n)} P(\norm{v_1}_{\dot{H}^{s} \cap \dot{H}^{k}(\mathbb{R}^n)},\norm{v_2}_{\dot{H}^{s} \cap \dot{H}^{k}(\mathbb{R}^n)})\norm{v_1-v_2}_{\dot{H}^{s} \cap \dot{H}^{k}(\mathbb{R}^n)}
\end{split}
\end{align}

for all $\abs{\varepsilon}, \abs{\kappa} \leq 1$ and all $u_1, u_2, u_3, v, v_1, v_2 \in \dot{H}_r^s(\R^n) \cap \dot{H}_r^k(\mathbb{R}^n)$ where $v, v_1$ and $v_2$ are real-valued and $P$ is a polynomial of degree $\leq 2k+1$.
\end{lemma}

\begin{framed}
    For the rest of the paper we will always assume the following condition onto the exponents $s$ and $k$
\begin{align}\label{Condition on exponents}
    \frac{n}{2} - 1 < s \leq \frac{n}{2} - 1 + \frac{1}{2n-2}, \quad k = n + 1.
\end{align}
\end{framed}

We will now state the main properties of the free part $\widetilde{L}$ and its linearization $\widetilde{L}_0$ appearing in \eqref{Perturbation equation}. A proof of these results can be found in \cite{AngKisSch26} Section 2.

\begin{proposition}\label{Proposition: Properties of L_0}
The operators $\widetilde{L}, \widetilde{L}_0 : C_{c,r}^{\infty}(\R^n) \subset \mc{H} \to \mc{H}$ are closable. Their closures $(L,\mc{D}(L))$ and $(L_0,\mc{D}(L_0))$ satisfy $\mc{D}(L) = \mc{D}(L_0)$, they are both self-adjoint, have compact resolvents and generate strongly continuous semigroups $(S(\tau))_{\tau \geq 0}$ and $(S_0(\tau))_{\tau \geq 0}$ on $\mc{H}$. The semigroup $(S(\tau))_{\tau\geq0}$ admits the explicit representation
\begin{align}\label{Representation of free semigroup}
    [S(\tau)f](x) = e^{-\frac{\tau}{2}} (H_{\alpha(\tau)} * f)(e^{-\frac{\tau}{2}}x), \quad x \in \R^n, 
\end{align}

where $H_{\alpha(\tau)}(x) = e^{-\frac{\abs{x}^2}{4\alpha(\tau)}}(4\pi\alpha(\tau))^{-\frac{n}{2}}$ and $\alpha(\tau) := 1 - e^{-\tau}$. 
    
Moreover, the restrictions $L^X := \restr{L}{X_s^k}$ and $L_0^X := \restr{L_0}{X_s^k}$ with domains $\mc{D}(L_0^X) = \mc{D}(L^X) = \{f \in \mc{D}(L) \cap X_s^k: Lf \in X_s^k\}$ now generate strongly continuous semigroups $(S^X(\tau))_{\tau \geq 0}$ and $(S_0^X(\tau))_{\tau \geq 0}$ on $X_s^k$ and they are given by the restrictions of $(S(\tau))_{\tau \geq 0}$ and $(S_0(\tau))_{\tau \geq 0}$ to $X_s^k$.

In addition, the following properties hold
\begin{itemize}
    \item[1.][compact perturbation] $S^X_0$ is a compact perturbation of $S^X$ in the sense that the right-hand side of
    \begin{align*}
        S^X_0(\tau) - S^X(\tau) = \int_0^{\tau} S^X_0(\tau - \tau') V_0(\psi_0) S^X(\tau') \, d\tau' 
    \end{align*}
    
    is compact for every $\tau \geq 0$.
    \item[2.][spectral properties] There exists an $\widetilde{\omega} > 0$ with
    \begin{align*}
        \sigma(L^X_0) \cap \{ \lambda \in \C : ~ \Re(\lambda) \geq -\widetilde{\omega} \} = \{1\}
    \end{align*}
    
    and $\lambda = 1$ is a simple eigenvalue whose eigenspace is spanned by $g_0 := \Lambda \psi_0$.
    \item[3.][decay on stable subspace] If we define the Riesz projection
    \begin{align*}
        P_0 := \frac{1}{2\pi i} \int_{\partial B_{1/2}(1)} R_{L^X_0}(\lambda) \, d\lambda
    \end{align*}
    
    then we have that for every $0 < \omega_0 < \widetilde{\omega}$ there holds
    \begin{align}\label{Decay on stable subspace}
        \norm{S^X_0(\tau)(I-P_0)f}_{s,k} \lesssim e^{-\omega_0 \tau} \norm{(I-P_0)f}_{s,k}
     \end{align}
    
    for all $\tau \geq 0$ and $f \in X_s^k$.
\end{itemize}
\end{proposition}

\section{Construction of the blowup solution}
To now prove the existence of a solution to \eqref{Perturbation equation} we proceed similar as in \cite{DonSchWit25}. We first observe that by Proposition \ref{Proposition: Properties of L_0} the operator $\widetilde{L}_0$ is closable in $\mc{H}$ and that the restriction of its closure onto $X_s^k$, namely $L_0^X$, is invertible.

We can therefore reformulate \eqref{Perturbation equation} as a fixed-point equation 
\begin{align}\label{Fixed point equation}
    \phi_{\varepsilon} = \left(-L_0^X\right)^{-1}\left(V_{\varepsilon}(\psi_0) \phi_{\varepsilon} + \widetilde{N}_{\varepsilon}(\phi_{\varepsilon}) + \mc{R}_{\varepsilon}(\psi_0)\right).
\end{align}

We will prove the necessary estimates for the involved operators via parameter depending Schauder-type estimates, see Lemma \ref{Lemma: Schauder}. For this we first need an auxiliary Lemma for $\eta_{\varepsilon}$ from \eqref{Nonlinearity}.

\begin{lemma}[Lemma 3.1 in \cite{DonSchWit25}]
Let $a,b,c \in \R$ and $\abs{\varepsilon} \leq 1$. Then
\begin{align}
    \eta_{\varepsilon}(a) = a^3\int_0^1\int_0^1\int_0^1 x^2y\,\eta_{\varepsilon}'''(axyz)\, dz \, dy \, dx,
\end{align}

\begin{align}\label{Second auxiliary equation}
    \eta_{\varepsilon}'(a) = a^2\int_0^1\int_0^1x\,\eta_{\varepsilon}'''(axy)\, dy \, dx,
\end{align}

and
\begin{align}\label{Third auxiliary equation}
    & \eta_{\varepsilon}(a+c)  - \eta_{\varepsilon}(a+b) - \eta_{\varepsilon}'(a)(c-b)\nonumber \\
    & = (c-b) \int_0^1\int_0^1\int_0^1 (b+x(c-b)) (a+y(b+x(c-b)))
    \eta_{\varepsilon}'''(z(a+y(b+x(c-b))))\, dz \, dy \, dx.
\end{align}

Moreover, for every $\ell \in \N_0$ there exists a constant $C_{\ell}$ such that
\begin{align}
    \abs{\eta_{\varepsilon}^{(\ell+3)}(x)-\eta_{\kappa}^{(\ell+3)}(y)} \leq C_{\ell}\left(\abs{\varepsilon-\kappa} +\abs{x-y}\right)
\end{align}

holds for every $x,y \in \R$ and $\abs{\varepsilon},\abs{\kappa} \leq 1$.
\end{lemma}

We can now prove the required norm estimates for the operators via the Schauder-type estimates from Lemma \ref{Lemma: Schauder}.

\begin{lemma}\label{Lemma: Operator estimates}
For any $\varepsilon \in \R$ with $|\varepsilon| \leq 1$ we have $V_{\varepsilon}(\psi_0): X_s^k \to X_s^k$ as well as $\mc{R}_{\varepsilon}(\psi_0) \in X_s^k$  with estimates
\begin{align*}
    \norm{V_{\varepsilon}(\psi_0)u - V_{\kappa}(\psi_0)u}_{s,k} & \lesssim \abs{\varepsilon-\kappa} \norm{u}_{s,k},\\
    \norm{\mc{R}_{\varepsilon}(\psi_0) - \mc{R}_{\kappa}(\psi_0)}_{s,k} & \lesssim \abs{\varepsilon-\kappa}
\end{align*}

for all $\abs{\varepsilon}, \abs{\kappa} \leq 1$  and all $u \in X_s^k$. For the nonlinearity $\widetilde{N}_{\varepsilon}: X_s^k \to X_s^k$ we obtain the following local Lipschitz bound
\begin{align}\label{Estimate for nonlinearity}
    \| \widetilde{N}_{\varepsilon}(u) - \widetilde{N}_{\kappa}(v) \|_{s,k} \lesssim \left(\norm{u}_{s,k}+\norm{v}_{s,k}\right)\norm{u-v}_{s,k} + \left(\norm{u}_{s,k}^2+ \norm{v}_{s,k}^2\right) \abs{\varepsilon-\kappa}
\end{align}

for all $\abs{\varepsilon},\abs{\kappa}\leq 1$ and all $u,v\in \mc{B}_{\delta} := \{u \in X_s^k : ~ \norm{u}_{s,k} \leq \delta \}$ for a fixed $0<\delta\leq1$.
\end{lemma}

\begin{proof}
We only prove the estimate for the nonlinearity. The estimates for the potential and the remainder term can be proven analogously.

We first prove
\begin{align}\label{Nonlinearity estimate 1}
    \lVert \widetilde{N}_{\varepsilon}(u) - \widetilde{N}_{\varepsilon}(v)\rVert_{s,k} \lesssim \left(\norm{u}_{s,k} + \norm{v}_{s,k}\right) \norm{u-v}_{s,k} 
\end{align}

for all $\abs{\varepsilon} \leq 1$ and all $u,v \in \mc{B}_{\delta}$ for $0 < \delta \leq 1$. We can first of all write with \eqref{Third auxiliary equation}
\begin{align*}
	& \widetilde{N}_{\varepsilon}(u)(\xi) - \widetilde{N}_{\varepsilon}(v)(\xi) = \\ & \frac{n-3}{\abs{\xi}^3}\Big(\eta_{\varepsilon}\big(|\xi|(\psi_0(\xi)+u(\xi))\big)-\eta_{\varepsilon}\big(|\xi|(\psi_0(\xi)+v(\xi))\big) - \eta_{\varepsilon}'(|\xi|\psi_0(\xi))|\xi|(u(\xi)-v(\xi)) \Big) \\
	& = (n-3)\int_{0}^{1}\int_{0}^{1}\int_{0}^{1} (u(\xi)-v(\xi))\big(v(\xi)+x(u(\xi)-v(\xi))\big) \\ & \cdot \left(\psi_0(\xi) + y\big(v(\xi)+x(u(\xi)-v(\xi))\big) \right)  \cdot \eta_{\varepsilon}'''\left( z |\xi| \big(\psi_0(\xi)+ y(v+x(u-v))\big) \right)dz \,dy\, dx. 
\end{align*}

Due to the fact that $F_{\varepsilon}(x) = \eta_{\varepsilon}'''(x)$ fulfills the assumptions of the Schauder type estimates from Lemma \ref{Lemma: Schauder} and the fact that we have the embedding $X_{s-1}^k \hookrightarrow X_s^k$ we obtain \eqref{Nonlinearity estimate 1}.
    
 To now prove 
\begin{align}\label{Nonlinearity estimate 2}
    \lVert\widetilde{N}_{\varepsilon}(u) - \widetilde{N}_{\kappa}(u)\rVert_{s,k} \lesssim \norm{u}_{s,k}^2 \abs{\varepsilon-\kappa}
\end{align}

for all $\abs{\varepsilon}, \abs{\kappa} \leq 1$ and all $u \in \mc{B}_{\delta}$ with $0 < \delta \leq 1$ we write with the help of \eqref{Third auxiliary equation}
\begin{align*}
	\widetilde{N}_{\varepsilon}(u)(\xi)-\widetilde{N}_{\kappa}(u)(\xi) & = (n-3) \,u(y)^2\int_{0}^{1} \int_{0}^{1}\int_{0}^{1}x\left( \psi_0(\xi) + yxu(\xi) \right)\\&  \cdot \left(\eta_{\varepsilon}'''\left( z |\xi| \big(\psi_0(\xi)+ yxu(\xi) \big) \right)-\eta_{\kappa}'''\left( z |\xi| \big(\psi_0(\xi)+ yxu(\xi)\big) \right)\right)dz \,dy\, dx.
\end{align*}

Also \eqref{Nonlinearity estimate 2} now follows from an application of Lemma \ref{Lemma: Schauder} and together with \eqref{Nonlinearity estimate 1} this implies \eqref{Estimate for nonlinearity}. As a consequence, and using that $\widetilde{N}_{\varepsilon}(0) = 0$, we conclude that $\widetilde{N}_{\varepsilon}$ defines a mapping from $X_s^k$ into itself.
    
\end{proof}

With this Lemma in hand we are now able to obtain a fixed point for \eqref{Fixed point equation}.

\begin{proposition}\label{Proposition: Contraction}
There exist $0 < \delta^* \leq 1$ and $0 < \varepsilon^* \leq 1$, where $\varepsilon^*$ is allowed to depend on $\delta^*$, such that for every $\abs{\varepsilon} \leq \varepsilon^*$ the map
\begin{align*}
    K_{\varepsilon}:\mc{B}_{\delta^*} \to \mc{B}_{\delta^*}, ~ K_{\varepsilon}(\phi) = \left(-L_0^X\right)^{-1}\left(V_{\varepsilon}(\psi_0)\phi + \mc{R}_{\varepsilon}(\psi_0) + \widetilde{N}_{\varepsilon}(\phi)\right)
\end{align*}

is a well-defined contraction. Additionally, $\delta^*$ and $\varepsilon^*$ can be chosen in such a way so that
\begin{align}\label{Lipschitz continuity for contraction}
    \norm{K_{\varepsilon}(\phi) - K_{\varepsilon}(\psi)}_{s,k} \leq \frac{1}{2} \norm{\phi-\psi}_{s,k}
\end{align}

holds for every $\abs{\varepsilon} \leq \varepsilon^*$ and all $\phi,\psi \in \mc{B}_{\delta^*}$.
\end{proposition}

\begin{proof}
For now we take arbitrary $0 < \delta^* \leq 1$ and $0 < \varepsilon^* \leq 1$. We first show that $K_{\varepsilon}: \mc{B}_{\delta^*} \to \mc{B}_{\delta^*}$ is well-defined for every $\abs{\varepsilon} \leq \varepsilon^*$. For this we first observe that due to the invertibility of $L_0^X$ in $X_s^k$ there exists a constant $C > 0$ (independent from $\varepsilon^*$ and $\delta^*$) with
\begin{align*}
    \norm{K_{\varepsilon}(\phi)}_{s,k} \leq C \left(\norm{V_{\varepsilon}(\psi_0)\phi}_{s,k} + \norm{\mc{R}_{\varepsilon}(\psi_0)}_{s,k} + \lVert \widetilde{N}_{\varepsilon}(\phi)\rVert_{s,k}\right)
\end{align*}

for all $\abs{\varepsilon} \leq \varepsilon^*$ and $\phi \in \mc{B}_{\delta^*}$. From Lemma \ref{Lemma: Operator estimates} above we then obtain additional constants $C_1, C_2$ and $C_3$ with 
\begin{align*}
    \norm{V_{\varepsilon}(\psi_0)\phi}_{s,k} + \norm{\mc{R}_{\varepsilon}(\psi_0)}_{s,k} + \lVert \widetilde{N}_{\varepsilon}(\phi) \rVert_{s,k} \leq C_1 \, \varepsilon^* \delta^* + C_2 \, \varepsilon^* + C_3 \left(\delta^*\right)^2.
\end{align*}

We now first choose $\delta^* < \left(C_3 C\right)^{-1}$ and then $\varepsilon^*$ so small so that we have \\ $C\left(C_1 \, \varepsilon^* + C_2 \left(\delta^*\right)^{-1} \varepsilon^* + C_3 \, \delta^*\right) \leq 1$, which shows that $K_{\varepsilon}: \mc{B}_{\delta^*} \to \mc{B}_{\delta^*}$ is well-defined.

To now show the contraction property we first obtain another constant $C_4$ from Lemma \ref{Lemma: Operator estimates} such that there holds
\begin{align*}
    \norm{K_{\varepsilon}(\phi) - K_{\varepsilon}(\psi)}_{s,k} & \leq C \left(\norm{V_{\varepsilon}(\psi_0)(\phi-\psi)}_{s,k} + \lVert \widetilde{N}_{\varepsilon}(\phi) - \widetilde{N}_{\varepsilon}(\psi) \rVert_{s,k}\right) \nonumber \\ & \leq C \left(C_1 \, \varepsilon^* + C_4 \, \delta^*\right) \norm{\phi-\psi}_{s,k}
\end{align*}

for every $\phi, \psi \in \mc{B}_{\delta^*}$ and $\abs{\varepsilon} \leq \varepsilon^*$. If we now choose $\varepsilon^*$ and $\delta^*$ even smaller so that $C \, C_1 \, \varepsilon^* + C_4 \, \delta^* \leq 1/2$ holds the claim follows.

\end{proof}

The next result shows that the above obtained fixed point solves the perturbation equation \eqref{Perturbation equation} pointwise and that it depends Lipschitz continuous on the parameter $\varepsilon$, which will be crucial for the stability analysis later on.

\begin{proposition}\label{Proposition: Existence of perturbation}
Take $\delta^*$ and $\varepsilon^*$ as in Proposition \ref{Proposition: Contraction}. Then there exists for every $\abs{\varepsilon} \leq \varepsilon^*$ a unique $\phi_{\varepsilon} \in \mc{B}_{\delta^*} \cap \mc{D}(L_0^X)$ which satisfies \eqref{Perturbation equation} pointwise.
We furthermore have 
\begin{align}\label{Lipschitz continuity for perturbation}
    \norm{\phi_{\varepsilon} - \phi_{\kappa}}_{s,k} \lesssim \abs{\varepsilon-\kappa}
\end{align}

for every $\abs{\varepsilon}, \abs{\kappa} \leq \varepsilon^*$.
\end{proposition}

\begin{proof}
From Proposition \ref{Proposition: Contraction} we immediately obtain for every $\abs{\varepsilon} \leq \varepsilon^*$ a unique $\phi_{\varepsilon} \in \mc{B}_{\delta^*} \cap \mc{D}(L_0^X)$ satisfying $\phi_{\varepsilon} = K_{\varepsilon}(\phi_{\varepsilon})$. That $\phi_{\varepsilon}$ now also satisfies \eqref{Perturbation equation} pointwise follows from the fact that $L_0^X$ acts as a differential operator on $\phi_{\varepsilon}$ due to the embedding $X_s^k(\R^n) \hookrightarrow C^2(\R^n)$. 
    
To show the Lipschitz continuity we take $\abs{\varepsilon}, \abs{\kappa} \leq \varepsilon^*$ and obtain with \eqref{Lipschitz continuity for contraction}
\begin{align*}
    \norm{\phi_{\varepsilon}-\phi_{\kappa}}_{s,k} = \norm{K_{\varepsilon}(\phi_{\varepsilon}) - K_{\kappa}(\phi_{\kappa})}_{s,k} \leq \frac{1}{2} \norm{\phi_{\varepsilon}-\phi_{\kappa}}_{s,k} + \norm{K_{\varepsilon}(\phi_{\kappa})-K_{\kappa}(\phi_{\kappa})}_{s,k}
\end{align*}

and therefore $\norm{\phi_{\varepsilon}-\phi_{\kappa}}_{s,k} \leq 2 \norm{K_{\varepsilon}(\phi_{\kappa})-K_{\kappa}(\phi_{\kappa})}_{s,k}$. \eqref{Lipschitz continuity for perturbation} now follows from the following estimate which holds for all $\phi \in \mc{B}_{\delta^*}$ due to Lemma \ref{Lemma: Operator estimates}
\begin{align*}
    & \norm{K_{\varepsilon}(\phi) - K_{\kappa}(\phi)}_{s,k} \\ & \lesssim \norm{\left(V_{\varepsilon}(\psi_0)-V_{\kappa}(\psi_0)\right)\phi}_{s,k} + \norm{\mc{R}_{\varepsilon}(\psi_0) - \mc{R}_{\kappa}(\psi_0)}_{s,k} + \lVert \widetilde{N}_{\varepsilon}(\phi) - \widetilde{N}_{\kappa}(\phi) \rVert_{s,k} \lesssim \abs{\varepsilon-\kappa}. \nonumber
\end{align*}
    
\end{proof}

From this we obtain the following result about the self-similar profile. 

\begin{proposition}\label{Proposition: properties of self-similar solution}
Let $\delta^* > 0$ and $\varepsilon^* > 0$ as in Proposition \ref{Proposition: Existence of perturbation} and let $\phi_{\varepsilon} \in \mc{B}_{\delta^*}$ be the solution to \eqref{Perturbation equation} for every $\abs{\varepsilon} \leq \varepsilon^*$. Then    
\begin{align*}
    \psi_{\varepsilon} := \psi_0 + \phi_{\varepsilon} \in X_s^k(\R^n) \hookrightarrow C_r^2(\R^n)
\end{align*}
    
is a smooth classical solution to \eqref{Static equation}. Moreover, it is even and we have
\begin{align*}
    \norm{\psi_{\varepsilon}-\psi_{\kappa}}_{s,k} \lesssim \abs{\varepsilon-\kappa}
\end{align*}
    
for all $\abs{\varepsilon},\abs{\kappa} \leq \varepsilon^*$. Furthermore, for every multi-index $\alpha \in \N_0^n$,
\begin{align}\label{Decay of psi_e}
    \abs{\partial^{\alpha}\psi_{\varepsilon}(y)} \lesssim_{\alpha} \langle y \rangle^{-1-\abs{\alpha}} \quad \text{and} \quad \abs{\partial^{\alpha} \left(\Lambda \psi_{\varepsilon}\right)(y)} \lesssim_{\alpha} \langle y \rangle^{-3-\abs{\alpha}} 
\end{align}

for every $y \in \R^n$ and $|\varepsilon| \leq \varepsilon^*$. The implicit constants can hereby be chosen uniformly with respect to $\varepsilon$. In addition,
\begin{align}\label{Lipschitz continuity of psi_e}
    \abs{\partial^{\alpha}(\psi_{\varepsilon}-\psi_{\kappa})(y)} \lesssim_{\alpha} |\varepsilon-\kappa| \langle y \rangle^{-1-\abs{\alpha}} \quad \text{and} \quad \abs{\partial^{\alpha} \left(\Lambda \psi_{\varepsilon} - \Lambda \psi_{\kappa}\right)(y)} \lesssim_{\alpha} |\varepsilon-\kappa|\langle y \rangle^{-3-\abs{\alpha}} 
\end{align}
for all $y \in \R^n$ and every $|\varepsilon|, |\kappa| \leq \varepsilon^*$.
    
\end{proposition}

\begin{proof}
The fact that $\psi_{\varepsilon}$ is a classical solution to \eqref{Static equation} follows from its construction and the embedding $X_s^k(\R^n) \hookrightarrow C_r^2(\R^n)$. The Lipschitz continuity with respect to the $X_s^k$-norm follows from the above Proposition. 

To simplify (and slight abuse of) the notation, we write for the remainder of the proof $\psi_{\varepsilon}(\rho) := \widetilde{\psi}_{\varepsilon}(\rho)$ for the radial profile of $\psi_{\varepsilon}$. This radial profile then satisfies the following second order ODE
\begin{align}\label{profile equation}
    0 = \psi_{\varepsilon}''(\rho) + \left(\frac{n-1}{\rho}-\frac{\rho}{2}\right) \psi_{\varepsilon}'(\rho) - \frac{1}{2} \psi_{\varepsilon}(\rho) + \frac{n-3}{\rho^3} \eta_{\varepsilon}(\rho \, \psi_{\varepsilon}(\rho)), \quad \text{for} \quad \rho > 0.
\end{align}

Since $\eta_{\varepsilon}(0) = \eta_{\varepsilon}'(0) = \eta_{\varepsilon}''(0) = 0$ we can write
\begin{align*}
    \eta_{\varepsilon}(z) = z^3 a_{\varepsilon}(z) \qquad \text{with} \qquad a_{\varepsilon}(z) = \frac{1}{2} \int_0^1 (1-t)^2 \eta_{\varepsilon}'''(tz) \, dt.
\end{align*}
$a_{\varepsilon}$ then depends smoothly on $\varepsilon$ so that we have for every $M > 0$ and $j \in \N_0$
\begin{align}\label{estimates for a_e}
    \sup_{|\varepsilon|\leq\varepsilon^*, \, |z| \leq M} |a_{\varepsilon}^{(j)}(z)| < \infty \qquad \text{as well as} \qquad \underset{|z| \leq M}{\sup} |a_{\varepsilon}^{(j)}(z) - a_{\kappa}^{(j)}(z)| \lesssim_{M,j} |\varepsilon - \kappa|
\end{align}
for every $|\varepsilon|, |\kappa| \leq \varepsilon^*$. We now write the profile equation \eqref{profile equation} as
\begin{align}\label{rewritten profile equation}
    \psi_{\varepsilon}'' + \frac{n-1}{\rho} \psi_{\varepsilon}' = F_{\varepsilon} \qquad \text{with} \quad F_{\varepsilon}(\rho) = \frac{\rho}{2} \psi_{\varepsilon}'(\rho) + \frac{1}{2} \psi_{\varepsilon}(\rho) - (n-3) \psi_{\varepsilon}(\rho)^3 a_{\varepsilon}(\rho \, \psi_{\varepsilon}(\rho))
\end{align}
and multiplication with $\rho^{n-1}$ gives
\begin{align*}
    \left(\rho^{n-1} \psi_{\varepsilon}'(\rho)\right)' = \rho^{n-1} F_{\varepsilon}(\rho).
\end{align*}
Since $\psi_{\varepsilon}'$ is bounded by the $C_2$-embedding of $X_s^k$ we have $\lim\limits_{\rho\to0} \rho^{n-1}\psi_{\varepsilon}'(\rho) = 0$ and integrating from 0 to $\rho$ therefore gives
\begin{align*}
    \psi_{\varepsilon}'(\rho) = \rho^{1-n} \int_0^{\rho} s^{n-1} F_{\varepsilon}(s) \, ds = \rho \int_0^1 t^{n-1} F_{\varepsilon}(t\rho) \, dt =: \rho \, \mc{T} F_{\varepsilon}(\rho).
\end{align*}
Substituting this into \eqref{rewritten profile equation} gives 
\begin{align}\label{rerewritten profile equation}
    \psi_{\varepsilon}'' = \mc{K} F_{\varepsilon}, \qquad \text{where} \qquad \mc{K} f := f - (n-1) \mc{T} f.
\end{align}
From this no follows the smoothness and evenness of $\psi_{\varepsilon}$ via a bootstrapping argument.
Differentiation under the integral gives for every $j \in \N_0$
\begin{align*}
    (\mc{T}f)^{(j)}(\rho) = \int_0^1 t^{n-1+j} f^{(j)}(t\rho) \, dt.
\end{align*}
Thus, for every $R > 0$ and $j \in \N_0$, we obtain
\begin{align*}
    \norm{\mc{T}f}_{C^j([0,R])} \leq \frac{1}{n+j} \norm{f}_{C^j([0,R])} \quad \text{and therefore} \quad \norm{\mc{K}f}_{C^j([0,R])} \lesssim_j \norm{f}_{C^j([0,R])}.
\end{align*}
We are now going to prove local bounds for the profile function. That is we prove inductively for every $m \in \N_0$ and $R > 0$
\begin{align}\label{local bounds}
    \sup_{|\varepsilon| \leq \varepsilon^*} \norm{\psi_{\varepsilon}}_{C^m([0,R])} < \infty \qquad \text{and} \qquad \norm{\psi_{\varepsilon} - \psi_{\kappa}}_{C^m([0,R])} \lesssim_{m,R} |\varepsilon - \kappa|
\end{align}
for all $|\varepsilon|, |\kappa| \leq \varepsilon^*$. From $X_s^k(\R^n) \hookrightarrow C_r^2(\R^n)$ the hypothesis immediately follows for $m \leq 2$. We now take an $m \geq 2$ and assume that the estimates from \eqref{local bounds} hold. Then, from the definition of $F_{\varepsilon}$ and \eqref{estimates for a_e} we obtain
\begin{align*}
    \sup_{|\varepsilon|\leq\varepsilon^*} \norm{F_{\varepsilon}}_{C^{m-1}([0,R])} < \infty \qquad \text{as well as} \qquad \norm{F_{\varepsilon} - F_{\kappa}}_{C^{m-1}([0,R])} \lesssim_{m,R} |\varepsilon - \kappa|
\end{align*}
Using \eqref{rerewritten profile equation} and the boundedness of $\mc{K}$, we obtain
\begin{align*}
    \sup_{|\varepsilon| \leq \varepsilon^*} \norm{\psi_{\varepsilon}''}_{C^{m-1}([0,R])} < \infty \qquad \text{and} \qquad \norm{\psi_{\varepsilon}'' - \psi_{\kappa}''}_{C^{m-1}([0,R])} \lesssim_{m,R} |\varepsilon - \kappa|,
\end{align*}
so that the local bounds from \eqref{local bounds} follow via induction.

To now prove the decay estimates at infinity, we first define
\begin{align*}
    p_{\varepsilon}(\rho) := \rho \, \psi_{\varepsilon}(\rho) \quad \text{and} \quad h_{\varepsilon}(\rho) := p_{\varepsilon}'(\rho) = \psi_{\varepsilon}(\rho) + \rho \, \psi_{\varepsilon}'(\rho),
\end{align*}
so that $h_{\varepsilon}$ is the radial profile of $2\Lambda\psi_{\varepsilon}$. We also define 
\begin{align*}
    q_{\varepsilon}(\rho) := w_{\varepsilon}(\rho) w_{\varepsilon}'(\rho)
\end{align*}
and since $\eta_{\varepsilon}(z) = z - q_{\varepsilon}(z)$ we obtain with \eqref{profile equation}
\begin{align}\label{equation for h_e}
    h_{\varepsilon}' + A \, h_{\varepsilon} = b_{\varepsilon} \quad \text{with} \quad A(\rho) := \frac{n-3}{\rho} - \frac{\rho}{2} \quad \text{and} \quad b_{\varepsilon}(\rho) := \frac{n-3}{\rho^2} q_{\varepsilon}(p_{\varepsilon}(\rho)).
\end{align}
If we introduce the following integrating factor
\begin{align*}
    \mu(\rho) = \rho^{n-3} e^{-\frac{\rho^2}{4}} 
\end{align*}
it satisfies $\mu'(\rho) = A(\rho) \mu(\rho)$ and therefore
\begin{align*}
    (\mu h_{\varepsilon}(\rho))'(\rho) = (n-3) \mu(\rho) \rho^{-2} q_{\varepsilon}(p_{\varepsilon}(\rho)).
\end{align*}
Since $\psi_{\varepsilon}$ and $\psi_{\varepsilon}'$ are bounded by the $C^2$-embedding, we know that $h_{\varepsilon}$ grows at most linearly so that we obtain $\lim\limits_{\rho\to\infty} \mu(\rho) h_{\varepsilon}(\rho) = 0$. Integrating from $\rho$ to $\infty$ now gives
\begin{align*}
    h_{\varepsilon}(\rho) = - \frac{n-3}{\mu(\rho)} \int_{\rho}^{\infty} \mu(s) s^{-2} q_{\varepsilon}(p_{\varepsilon}(s)) \, ds
\end{align*}
and since $w_{\varepsilon}$ is smooth and periodic, one has for every $j \in \N_0$
\begin{align*}
    \sup_{|\varepsilon| \leq \varepsilon^*} \|q_{\varepsilon}^{(j)}\|_{L^{\infty}(\R)} < \infty.
\end{align*}
Thus, for $\rho \geq 1$, we obtain
\begin{align}\label{decay h_e}
    |h_{\varepsilon}(\rho)| \lesssim \rho^{3-n} e^{\frac{\rho^2}{4}} \int_{\rho}^{\infty} s^{n-5} e^{-\frac{s^2}{4}} \, ds \lesssim  \rho^{-3}. 
\end{align}
Since $p_{\varepsilon}' = h_{\varepsilon}$ it follows for $\rho_2 \geq \rho_1 \geq 1$
\begin{align*}
    |p_{\varepsilon}(\rho_2) - p_{\varepsilon}(\rho_1)| \leq \int_{\rho_1}^{\infty} |h_{\varepsilon}(s)| \, ds \lesssim \rho_1^{-2},  
\end{align*}
so that $\lim\limits_{\rho \to\infty} p_{\varepsilon}(\rho)$ exists. Since $p_{\varepsilon}(1) = \psi_{\varepsilon}(1)$ is uniformly bounded with respect to $\varepsilon$ we obtain with the estimate from above
\begin{align}\label{uniform bound for p_e}
    |p_{\varepsilon}(\rho)| \lesssim 1 \quad \text{and therefore} \quad |\psi_{\varepsilon}(\rho)| \lesssim \rho^{-1}.
\end{align}

We now prove
\begin{align*}
    |h_{\varepsilon}^{(m)}(\rho)| \lesssim_m \rho^{-m-3} \qquad \text{for all} \quad \rho \geq 1 
\end{align*}
via induction on $m \in \N_0$. And since the case $m = 0$ has already been proven in \eqref{decay h_e}, we can assume that $|h_{\varepsilon}^{(j)}(\rho)| \lesssim_j \rho^{-j-3}$ for $0 \leq j \leq m-1$. Since $p_{\varepsilon}' = h_{\varepsilon}$ we have 
\begin{align*}
    |p_{\varepsilon}^{(j)}(\rho)| = |h_{\varepsilon}^{(j-1)}(\rho)| \lesssim_j \rho^{-j-2} \qquad \text{for} \quad 1 \leq j \leq m.
\end{align*}
From the chain rule and the uniform boundedness of $q_{\varepsilon}$ and its derivatives we obtain
\begin{align*}
    \left|\frac{d^j}{d\rho^j}q_{\varepsilon}(p_{\varepsilon}(\rho))\right| \lesssim_j \rho^{-j-2} \qquad \text{for} \quad 1 \leq j \leq m \quad \text{and therefore} \quad |b_{\varepsilon}^{(m)}(\rho)| \lesssim \rho^{-m-2}.
\end{align*}
If we now differentiate \eqref{equation for h_e} $m$ times we get
\begin{align}\label{mth order equation for h_e}
    h_{\varepsilon}^{(m+1)} + A h_{\varepsilon}^{(m)} = b_{\varepsilon}^{(m)} - \sum_{j=1}^m \begin{pmatrix}
        m \\ j
    \end{pmatrix} A^{(j)} h_{\varepsilon}^{(m-j)} =: R_{\varepsilon,m}.
\end{align}
Now $A'(\rho) = - \frac{n-3}{\rho^2} - \frac{1}{2}$ is bounded for $\rho \geq 1$, whereas for $j \geq 2$ we have $A^{(j)}(\rho) = \mc{O}(\rho^{-j-1})$. The induction hypothesis therefore gives
\begin{align*}
    |R_{\varepsilon,m}(\rho)| \lesssim_m \rho^{-m-2}.
\end{align*}
Due to the formula of the integrating factor we can rewrite \eqref{mth order equation for h_e} as
\begin{align*}
    \left(\mu h_{\varepsilon}^{(m)}\right)' = \mu R_{\varepsilon,m}
\end{align*}
and from \eqref{equation for h_e} and its successive derivatives we obtain inductively that all of the derivatives of $h_{\varepsilon}$ have at most polynomial growth, so that we get $\mu(\rho) h_{\varepsilon}^{(m)}(\rho) \to 0$ as $\rho \to \infty$. From this it follows that
\begin{align*}
    h_{\varepsilon}^{(m)}(\rho) = -\frac{1}{\mu(\rho)} \int_{\rho}^{\infty} \mu(s) R_{\varepsilon,m}(s) \, ds
\end{align*}
and consequently
\begin{align*}
    |h_{\varepsilon}^{(m)}(\rho)| \lesssim \rho^{3-n} e^{\frac{\rho^2}{4}} \int_{\rho}^{\infty} s^{n-m-5} e^{-\frac{s^2}{4}} \, ds \lesssim \rho^{-m-3},
\end{align*}
which closes the induction. Now since
\begin{align*}
    h_{\varepsilon}^{(m)}(\rho) = \rho \psi_{\varepsilon}^{(m+1)}(\rho) + (m+1) \psi_{\varepsilon}^{(m)}(\rho),
\end{align*}
we obtain
\begin{align*}
    \psi_{\varepsilon}^{(m+1)}(\rho) = \frac{h_{\varepsilon}^{(m)}(\rho)}{\rho} - \frac{m+1}{\rho} \psi_{\varepsilon}^{(m)}(\rho).
\end{align*}
And since we have shown that $|\psi_{\varepsilon}(\rho)| \lesssim \rho^{-1}$ holds, an induction then yields
\begin{align*}
    |\psi_{\varepsilon}^{(m)}(\rho)| \lesssim \rho^{-m-1} \qquad \text{for every} \quad m \in \N_0 \quad \text{and} \quad \rho \geq 1.
\end{align*}
Now we want to prove the parameter-dependent decay estimates from \eqref{Lipschitz continuity of psi_e}. For this we define
\begin{align*}
    P_{\varepsilon,\kappa}(\rho) := p_{\varepsilon}(\rho) - p_{\kappa}(\rho) \quad \text{and} \quad H_{\varepsilon,\kappa}(\rho) := h_{\varepsilon}(\rho) - h_{\kappa}(\rho),
\end{align*}
so that we again have $P_{\varepsilon,\kappa}' = H_{\varepsilon,\kappa}$. Subtracting \eqref{equation for h_e} for $h_{\varepsilon}$ and $h_{\kappa}$ yields
\begin{align*}
    H_{\varepsilon,\kappa}' + A H_{\varepsilon,\kappa} = \frac{n-3}{\rho^2} Q_{\varepsilon,\kappa} \qquad \text{with} \quad Q_{\varepsilon,\kappa}(\rho) := q_{\varepsilon}(p_{\varepsilon}(\rho)) - q_{\kappa}(p_{\kappa}(\rho)).
\end{align*}
Since $q_{\varepsilon}$ and its derivatives depend Lipschitz continuously on $\varepsilon$, one has
\begin{align*}
    |q_{\varepsilon}(x) - q_{\kappa}(y)| \lesssim |\varepsilon-\kappa| + |x-y|
\end{align*}
for all $x,y \in \R$ and therefore $|Q_{\varepsilon,\kappa}(\rho)| \lesssim |\varepsilon-\kappa| + |P_{\varepsilon,\kappa}(\rho)|$. From the formula of the integrating factor we again obtain
\begin{align*}
    H_{\varepsilon,\kappa}(\rho) = - \frac{n-3}{\mu(\rho)} \int_{\rho}^{\infty} \mu(s) s^{-2} Q_{\varepsilon,\kappa}(s) \, ds.
\end{align*}
For now take $R > 1$, to be fixed later, and define
\begin{align*}
    M_{\varepsilon,\kappa}(R) := \sup_{\rho \geq R} |P_{\varepsilon,\kappa}(\rho)|.
\end{align*}
Due to \eqref{uniform bound for p_e} we obtain $M_{\varepsilon,\kappa}(R) < \infty$ uniformly in $\varepsilon$ and $\kappa$. Since we have for $\rho \geq R$
\begin{align*}
    |Q_{\varepsilon,\kappa}(\rho)| \lesssim |\varepsilon-\kappa| + |P_{\varepsilon,\kappa}(\rho)| \leq |\varepsilon-\kappa| + M_{\varepsilon,\kappa}(R), 
\end{align*}
we similarly obtain the following estimate as we did for $h_{\varepsilon}$ for $\rho \geq R$
\begin{align}\label{estimate for H_ek}
    |H_{\varepsilon,\kappa}(\rho)| \lesssim \rho^{3-n} e^{\frac{\rho^2}{4}} \int_{\rho}^{\infty} s^{n-5} e^{-\frac{s^2}{4}} |Q_{\varepsilon,\kappa}(s)| \, ds \lesssim \left(|\varepsilon-\kappa| + M_{\varepsilon,\kappa}(R)\right)\rho^{-3}.
\end{align}
Moreover, due to $P_{\varepsilon,\kappa}' = H_{\varepsilon,\kappa}$
\begin{align*}
    P_{\varepsilon,\kappa}(\rho) = P_{\varepsilon,\kappa}(R) + \int_R^{\rho} H_{\varepsilon,\kappa}(s)\,ds 
\end{align*}
and from the local bounds \eqref{local bounds} we obtain
\begin{align*}
    |P_{\varepsilon,\kappa}(R)| = R |\psi_{\varepsilon}(R) - \psi_{\kappa}(R)| \leq C_R |\varepsilon-\kappa|. 
\end{align*}
It follows
\begin{align*}
    M_{\varepsilon,\kappa}(R) \leq |P_{\varepsilon,\kappa}(R)| + \int_R^{\infty} |H_{\varepsilon,\kappa}(s)| \, ds \lesssim C_R |\varepsilon-\kappa| + R^{-2}\left(|\varepsilon-\kappa| + M_{\varepsilon,\kappa}(R)\right).
\end{align*}
We now choose $R>1$ sufficiently large (independently of $\varepsilon$ and $\kappa$) such that the coefficient on the right-hand side is strictly smaller than 1. Absorbing this term then gives $M_{\varepsilon,\kappa}(R) \lesssim |\varepsilon-\kappa|$, so that we obtain from \eqref{estimate for H_ek}
\begin{align*}
    |H_{\varepsilon,\kappa}(\rho)| \lesssim |\varepsilon-\kappa|\rho^{-3} \qquad \text{for} \quad \rho \geq R.
\end{align*}
The estimates for the derivatives of $H_{\varepsilon,\kappa}$ and $P_{\varepsilon,\kappa}$ now follow analogously via induction as for $h_{\varepsilon}$ and $p_{\varepsilon}$. Overall, we obtain \eqref{Lipschitz continuity of psi_e}, which finishes the proof.

\end{proof}

We are now able to prove our first main theorem.

\begin{proof}[Proof of Theorem \ref{Theorem: Blowup Solution}]

We define $\widetilde{f}_{\varepsilon}(\rho) = \rho \, \widetilde{\psi}_{\varepsilon}(\rho)$, where $\widetilde{\psi}_{\varepsilon}$ is the radial profile from the above constructed $\psi_{\varepsilon}$ solving \eqref{Perturbation equation}. From this we immediately obtain that $\widetilde{f}_{\varepsilon}$ is a smooth and odd function with $\widetilde{f}_{\varepsilon}(0) = 0$. Furthermore, the limit $L := \lim\limits_{\rho\to\infty} \widetilde{f}_{\varepsilon}(\rho)$ exists and there exists for every $k \in \N$ a constant $C_k > 0$ with
\begin{align*}
    \abs{\widetilde{f}_{\varepsilon}^{(k)}(\rho)} \leq C_{\varepsilon,k} \langle \rho \rangle^{-2-k} \qquad \text{for all} \quad \rho > 0.
\end{align*}

We now show $0 < \widetilde{f}_{\varepsilon}(\rho) < \pi$ for every $\rho > 0$. For this we first observe that $\widetilde{f}_{\varepsilon}$ solves the following profile equation
\begin{align}\label{ODE for self-similar profile}
    0 = f''(\rho) + \left(\frac{d-1}{\rho} - \frac{\rho}{2}\right) f'(\rho) - \frac{d-1}{\rho^2} w_{\varepsilon}(f(\rho)) \, w_{\varepsilon}'(f(\rho)).
\end{align}
We define
\begin{align*}
    \mc{E}(\rho) := \frac{1}{2} \rho^2 \widetilde{f}_{\varepsilon}'(\rho)^2 - \frac{d-1}{2} w_{\varepsilon}(\widetilde{f}_{\varepsilon}(\rho))^2
\end{align*}
and differentiation with respect to $\rho$ gives
\begin{align*}
    \mc{E}'(\rho) = \rho \widetilde{f}_{\varepsilon}'(\rho)^2 + \rho^2 \widetilde{f}_{\varepsilon}'(\rho) \widetilde{f}_{\varepsilon}''(\rho) - (d-1) w_{\varepsilon}(\widetilde{f}_{\varepsilon}(\rho)) w_{\varepsilon}'(\widetilde{f}_{\varepsilon}(\rho)) \widetilde{f}_{\varepsilon}'(\rho).
\end{align*}
If we now multiply the profile equation \eqref{ODE for self-similar profile} with $\rho^2 \widetilde{f}_{\varepsilon}'(\rho)$ and plug this into the above equation we obtain
\begin{align*}
    \mc{E}'(\rho) = \rho \left(\frac{\rho^2}{2} - (d-2)\right) \widetilde{f}_{\varepsilon}'(\rho)^2.
\end{align*}
If we now set $\rho_* := \sqrt{2(d-2)}$ we get $\mc{E}'(\rho) \leq 0$ for $0 \leq \rho \leq \rho_*$ and $\mc{E}'(\rho) \geq 0$ for $\rho \geq \rho_*$. 

Since $\mc{E}(0) = 0$ we therefore have $\mc{E}(\rho) \leq 0$ for all $0 \leq \rho \leq \rho_*$ and due to $\widetilde{f}_{\varepsilon}'(\rho) = \mc{O}(\rho^{-3})$ we get
\begin{align*}
    \lim_{\rho\to\infty} \mc{E}(\rho) = - \frac{d-1}{2} w_{\varepsilon}(L)^2 \leq 0.
\end{align*}
Since $\mc{E}$ is nondecreasing on $[\rho_*,\infty)$ and has a nonpositive limit, it follows $\mc{E}(\rho) \leq 0$ for $\rho \geq \rho_*$ and therefore
\begin{align}\label{energy nonpositive}
    \mc{E}(\rho) \leq 0 \quad \text{for all} \quad \rho \geq 0.
\end{align}
We now define
\begin{align*}
    \rho_0 := \inf\{\rho > 0 : ~ \widetilde{f}_{\varepsilon}(\rho) \notin (0,\pi)\} > 0.
\end{align*}
Due to the continuity of $\widetilde{f}_{\varepsilon}$ we must have $\widetilde{f}_{\varepsilon}(\rho_0) \in \{0,\pi\}$ and therefore
\begin{align*}
    \mc{E}(\rho_0) = \frac{1}{2} \rho_0^2 \widetilde{f}_{\varepsilon}'(\rho_0)^2 \geq 0.
\end{align*}
But due to \eqref{energy nonpositive} we must have $\mc{E}(\rho_0) = 0$ and therefore $\widetilde{f}_{\varepsilon}'(\rho_0) = 0$. Thus, at $\rho_0$, we have either
\begin{align*}
    \widetilde{f}_{\varepsilon}(\rho_0) = 0, \quad \widetilde{f}_{\varepsilon}'(\rho_0) = 0 \quad \text{or} \quad \widetilde{f}_{\varepsilon}(\rho_0) = \pi, \quad \widetilde{f}_{\varepsilon}'(\rho_0) = 0.
\end{align*}
But since the profile equation \eqref{ODE for self-similar profile} is a regular second-order ODE at $\rho_0 > 0$ the profile $\widetilde{f}_{\varepsilon}$ would have to equal $k\pi$ with $k \in \{0,1\}$ on the entire existence interval by uniqueness, which would be a contradiction to $\widetilde{f}_{\varepsilon}'(0) \neq 0$. If we now set
\begin{align*}
    u_{\varepsilon}^T(t,r) = \widetilde{f}_{\varepsilon}\left(\frac{r}{\sqrt{T-t}}\right) 
\end{align*}

we get for the gradient
\begin{align*}
    \restr{\partial_r u_{\varepsilon}^T(t,r)}{r=0} = \left(T-t\right)^{-1/2} \widetilde{f}_{\varepsilon}'(0) = \left(T-t\right)^{-1/2} \left(\widetilde{\psi}_0(0) + \widetilde{\phi}_{\varepsilon}(0)\right),
\end{align*}

so that the Theorem follows if we initially choose $\delta > 0$ so small so that we have $\abs{\widetilde{\phi}_{\varepsilon}(0)} < \widetilde{\psi}_0(0)$, which can be done due to the embedding $X_s^k(\R^n) \hookrightarrow L^{\infty}(\R^n)$. 
    
\end{proof}

\section{Stability analysis}\label{Section: Stability Analysis}

In this section we proceed to investigate the stability of the blowup solution $u_{\varepsilon}^T \in C^{\infty}([0,T)\times [0,\infty))$ constructed in Theorem \ref{Theorem: Blowup Solution}. We proceed similar to \cite{DonSchWit25} Section 4. For this, we consider $\widetilde{v}(t,r) = r^{-1} \widetilde{u}(t,r)$ and study the following Cauchy problem on $[0,\infty) \times \R^n$
\begin{align}\label{Cauchy problem stability analysis}
    \begin{cases}
    \partial_t v - \Delta_x v = \frac{n-3}{\abs{x}^3}\left(\abs{x}v - w_{\varepsilon}(\abs{x}v)w_{\varepsilon}'(\abs{x}v)\right) \\
    v(0,x) = v_{\varepsilon}^1(0,x) + \varphi_0(x). 
    \end{cases}
\end{align}
         
Here, $v_{\varepsilon}^1$ denotes the corresponding blowup solution $v_{\varepsilon}^T$ with blowup time $T=1$, i.e.
\begin{align*}
    v_{\varepsilon}^T(t,x) = \frac{1}{\sqrt{T-t}} \psi_{\varepsilon}\left(\frac{x}{\sqrt{T-t}} \right), \quad \psi_{\varepsilon}(y) = \abs{y}^{-1} \widetilde{f}_{\varepsilon}(|y|)
\end{align*}

and $\varphi_0$ is a radial real-valued Schwartz function with for now small $X_s^k$-norm.
    
Again, introducing similarity coordinates
\begin{align*}
    \tau = \log\left(\frac{T}{T-t}\right) \quad \text{and} \quad y = \frac{x}{\sqrt{T-t}}
\end{align*}

\eqref{Cauchy problem stability analysis} transforms into \eqref{Evolution equation} and the blowup solution $v_{\varepsilon}^T$ corresponds to the static solution $\psi_{\varepsilon}$ of Eq.~\eqref{Evolution equation}. The Cauchy problem \eqref{Cauchy problem stability analysis} can therefore be reformulated as
\begin{align}\label{Cauchy problem for stability analysis operator formulation}
    \begin{cases}
        \partial_{\tau}\psi(\tau) = \widetilde{L} \psi(\tau)+ N_{\varepsilon}(\psi(\tau)), \\
        \psi(0) = \psi_{\varepsilon}^T + \varphi_0^T,
     \end{cases}
\end{align}

where we use the convention that a superscript $T$ denotes the rescaling $f^T := \sqrt{T} f(\sqrt{T}\cdot)$ . $\widetilde{L}$ still denotes the heat operator in similarity coordinates and the nonlinearity $N_{\varepsilon}$ is the one from \eqref{Nonlinearity}.

To analyze the dynamics near $\psi_{\varepsilon}$ we decompose the solution as a perturbation of the static profile by writing $\psi(\tau) = \psi_{\varepsilon} + \phi_{\varepsilon}(\tau)$ and substitute this ansatz into the equation. This yields an evolution equation for the perturbation $\phi_{\varepsilon}$
\begin{align*}
    \partial_{\tau}\phi_{\varepsilon}(\tau)=\widetilde{L}\phi_{\varepsilon}(\tau) + N_{\varepsilon}(\psi_{\varepsilon} + \phi_{\varepsilon}(\tau)) - N_{\varepsilon}(\psi_{\varepsilon}).
\end{align*}

Expanding the nonlinearity around $\psi_{\varepsilon}$ we obtain the central evolution equation of this section
\begin{align}\label{Cauchy problem for perturbation stability analysis}
    \begin{cases}
        \partial_{\tau}\phi_{\varepsilon}(\tau) = \widetilde{L}_{\varepsilon} \phi_{\varepsilon}(\tau) +  \widehat{N}_{\varepsilon}(\phi_{\varepsilon}(\tau)),\\
        \phi_{\varepsilon}(0) = \varphi_0^T + \psi_{\varepsilon}^T - \psi_{\varepsilon}. 
    \end{cases}
\end{align}

Here, the linearized operator $\widetilde{L}_{\varepsilon} := \widetilde{L} + L_{\varepsilon}'$ consists of the heat operator $\widetilde{L}$ in similarity coordinates from \eqref{Evolution equation} and a potential term $L_{\varepsilon}'$ given by 
\begin{align*}
    L_{\varepsilon}' u = V_{\varepsilon}(\psi_{\varepsilon}) \, u \quad \text{ for } ~ V_{\varepsilon}(\psi_{\varepsilon})(y) = \frac{n-3}{\abs{y}^2}\eta_{\varepsilon}'(\abs{y}\psi_{\varepsilon}(y)).
\end{align*}

The nonlinear remainder $\widehat{N}_{\varepsilon}$ consists of the higher-order terms and is given by
\begin{align}\label{Definition: Nonlinearity in stability analysis}
    \widehat{N}_{\varepsilon}(u)(y)= \frac{n-3}{\abs{y}^3}\left(\eta_{\varepsilon}(\abs{y}(\psi_{\varepsilon}(y)+u(y)))-\eta_{\varepsilon}(\abs{y}\psi_{\varepsilon}(y))-\eta_{\varepsilon}'(\abs{y}\psi_{\varepsilon}(y))\abs{y}u(y)\right).
\end{align}

We note that the dependence on the blowup time $T$ appears only in the initial data. Furthermore, in the unperturbed case $\varepsilon = 0$, the operator $\widetilde{L}_{\varepsilon}$ coincides with the one from Proposition \ref{Proposition: Properties of L_0}. 

\subsection{The linearized operator for small parameters}

In this section we study the linear operator $\widetilde{L}_{\varepsilon}$ and its spectral properties. We begin by showing that $\widetilde{L}_{\varepsilon}$ is closable in $\mc{H}$ and that its closure $(L_{\varepsilon}, \mc{D}(L_{\varepsilon}))$ generates a strongly continuous semigroup $(S_{\varepsilon}(\tau))_{\tau\geq0}$ on $\mc{H}$. If we restrict this operator onto $X_s^k$ it likewise generates a strongly continuous semigroup on this space, which is obtained from the restriction of $(S_{\varepsilon}(\tau))_{\tau\geq0}$ onto $X_s^k$. 

We then turn to the spectral analysis of the restricted operator $L_{\varepsilon}^X$. We first show that $\lambda = 1$ is a simple eigenvalue, see Lemma \ref{Lemma: Eigenfunction}, corresponding to the expected instability associated to the time translation symmetry. We then show in Proposition \ref{Proposition: Stability of spectrum of Le} that this is in fact the only unstable spectral point of $L_{\varepsilon}^X$. A key ingredient in this analysis is the Lipschitz-dependency of $L_{\varepsilon}^X$ with respect to the parameter $\varepsilon$ proved in Proposition \ref{Proposition: Lipschitz continuity of full linear part}, which relies on the Lipschitz-continuity of the self-similar solution $\psi_{\varepsilon}$.

\begin{proposition}
For all $\abs{\varepsilon} \leq \varepsilon^*$ the operator $(\widetilde{L}_{\varepsilon}, \mc{D}(\widetilde{L}_{\varepsilon}))$ is closable in $\mc{H}$ and its closure $L_{\varepsilon}: \mc{D}(L_{\varepsilon}) \subset \mc{H} \to \mc{H}$ satisfies $\mc{D}(L_{\varepsilon}) = \mc{D}(L)$, is self-adjoint, has compact resolvent and generates a strongly continuous semigroup $(S_{\varepsilon}(\tau))_{\tau\geq0}$ of bounded operators on $\mc{H}$.
\end{proposition}

\begin{proof}
We only have to show that $L_{\varepsilon}': \mc{H} \to \mc{H}$ is bounded. The bounded perturbation theorem, see \cite{EngNag00} Theorem III.1.3, then implies that $\widetilde{L}_{\varepsilon}$ is closable in $\mc{H}$, that its closure has domain $\mc{D}(L_{\varepsilon}) = \mc{D}(L)$ and that it generates a strongly continuous semigroup on $\mc{H}$. Moreover, the compactness of the resolvent is preserved under bounded perturbations as well, see \cite{EngNag00} Proposition III.1.12, and the self-adjointness then follows from Kato-Rellich's Theorem.

To show the boundedness of $L_{\varepsilon}'$ in $\mc{H}$ we write the potential in the following way using \eqref{Second auxiliary equation}
\begin{align*}
    V_{\varepsilon}(\psi_{\varepsilon})(\xi) = (n-3)\psi_{\varepsilon}(\xi)^2\int_0^1x\int_0^1\eta_{\varepsilon}'''(\abs{\xi}\psi_{\varepsilon}(\xi)xy)dy\,dx.
\end{align*}

Therefore, $V_{\varepsilon}(\psi_{\varepsilon})$ behaves in leading order like $\psi_{\varepsilon}^2$ and therefore satisfies
\begin{align*}
    \abs{\partial^{\alpha} V_{\varepsilon}(\psi_{\varepsilon})(y)} \lesssim_{\alpha} \langle y \rangle^{-2-\abs{\alpha}}
\end{align*}

for every $\alpha \in \N_0^n$ and $y \in \R^n$. With Lemma 2.1 from \cite{DonSchWit25} we infer that $V_{\varepsilon}(\psi_{\varepsilon})$ belongs to $X_s^k(\R^n)$ and is therefore bounded, which then shows the boundedness of $L_{\varepsilon}'$ in $\mc{H}$.

\end{proof}

We now use the embedding $X_s^k(\R^n) \hookrightarrow \mc{H}$ to show that these operators also generate strongly continuous semigroups in $X_s^k(\R^n)$, which will be  the semigroups in $\mc{H}$ restricted onto $X_s^k(\R^n)$.

\begin{proposition}
For all $\abs{\varepsilon} \leq \varepsilon^*$ the restrictions $L^X := \restr{L}{X_s^k(\R^n)}$ and $L_{\varepsilon}^X := \restr{L_{\varepsilon}}{X_s^k(\R^n)}$ with domains $\mc{D}(L_{\varepsilon}^X) = \mc{D}(L^X) = \{f \in \mc{D}(L) \cap X_s^k(\R^n): Lf \in X_s^k(\R^n)\}$ generate strongly continuous semigroups $(S^X(\tau))_{\tau\geq0}$ and $(S_{\varepsilon}^X(\tau))_{\tau\geq0}$ on $X_s^k(\R^n)$ respectively and they are given by the restrictions of $(S(\tau))_{\tau\geq0}$ and $(S_{\varepsilon}(\tau))_{\tau\geq0}$ onto $X_s^k(\R^n)$ respectively.
\end{proposition}

\begin{proof}
From \cite{AngKisSch26} we infer that $(S(\tau))_{\tau\geq0}$ is also a strongly continuous semigroup on $X_s^k(\R^n)$. Since this space is invariant under $(S(\tau))_{\tau\geq0}$ we know from \cite{EngNag00}, II.2.3 that $L^X = \restr{L}{X_s^k(\R^n)}$ with domain $\mc{D}(L^X) = \{f \in \mc{D}(L) \cap X_s^k(\R^n): Lf \in X_s^k(\R^n)\}$ generates a strongly continuous semigroup $(S^X(\tau))_{\tau\geq0}$ which is given by the restriction of $(S(\tau))_{\tau\geq0}$ onto $X_s^k(\R^n)$.

Due to the Algebra property of $X_s^k(\R^n)$ we have
\begin{align*}
    \norm{L_{\varepsilon}'f}_{s,k} \lesssim \norm{V_{\varepsilon}(\psi_{\varepsilon})}_{s,k} \norm{f}_{s,k}
\end{align*}

for all $f \in X_s^k(\R^n)$ and therefore we have by the bounded perturbation theorem that $L_{\varepsilon}^X = \restr{L_{\varepsilon}}{X_s^k(\R^n)}$ has domain $\mc{D}(L_{\varepsilon}^X) = \mc{D}(L^X)$ and also generates a strongly continuous semigroup $(S_{\varepsilon}^X(\tau))_{\tau\geq0}$, which is given by the restriction of $(S_{\varepsilon}(\tau))_{\tau\geq0}$ onto $X_s^k(\R^n)$.
    
\end{proof}

\begin{proposition}
For every $\varepsilon \in \R$ with $\abs{\varepsilon} \leq \varepsilon^*$ the operator $L_{\varepsilon}' : X_s^k \to X_s^k$ depends Lipschitz continuous on the parameter $\varepsilon$ in the sense that 
\begin{align*}
    \norm{L_{\varepsilon}'u - L_{\kappa}'u}_{s,k} \lesssim \abs{\varepsilon-\kappa} \norm{u}_{s,k}
\end{align*}

holds for all $u \in X_s^k$ and all $\abs{\varepsilon},\abs{\kappa}\leq \varepsilon^*$.
\end{proposition}

\begin{proof}
Again we write
\begin{align*}
    V_{\varepsilon}(\psi_{\varepsilon})(\xi) = (n-3)\psi_{\varepsilon}(\xi)^2\int_0^1x\int_0^1\eta_{\varepsilon}'''(\abs{\xi}\psi_{\varepsilon}(\xi)xy)dy\,dx
\end{align*}

so that the Lipschitz continuity follows from the Schauder estimate from Lemma \ref{Lemma: Schauder} and the Lipschitz continuity of $\psi_{\varepsilon}$.

\end{proof}

From this we obtain the Lipschitz dependency of the whole linear operator.

\begin{proposition}\label{Proposition: Lipschitz continuity of full linear part}
For every $\varepsilon \in \R$ with $\abs{\varepsilon}\leq \varepsilon^*$ the operator $L_{\varepsilon}^X$ satisfies 
\begin{align*}
    \norm{L_{\varepsilon}^X u - L_{\kappa}^X u}_{s,k} \lesssim \abs{\varepsilon-\kappa} \norm{u}_{s,k}
\end{align*}

for all $u \in \mc{D}(L^X) = \mc{D}(L_{\varepsilon}^X) = \mc{D}(L_{\kappa}^X)$ and all $\abs{\varepsilon},\abs{\kappa}\leq \varepsilon^*$.
\end{proposition}

As in the unperturbed case $\varepsilon = 0$ the operator $L_{\varepsilon}^X$ has an unstable mode at $\lambda = 1$ generated by the time translation symmetry. More precisely, we have the following result.

\begin{lemma}\label{Lemma: Eigenfunction}
For every $\abs{\varepsilon} \leq \varepsilon^*$ we have 
\begin{align*}
    L_{\varepsilon}^X \, g_{\varepsilon} = g_{\varepsilon},
\end{align*}
    
where $g_{\varepsilon} := \Lambda \psi_{\varepsilon} \in C^{\infty}_r(\R^n) \cap \mc{D}(L_{\varepsilon}^X)$. Additionally, the eigenfunction $g_{\varepsilon}$ depends Lipschitz continuously on the parameter $\varepsilon$ in the sense that
\begin{align*}
    \norm{g_{\varepsilon} - g_{\kappa}}_{s,k} \lesssim |\varepsilon - \kappa|
\end{align*} 
holds for all $|\varepsilon|, |\kappa| \leq \varepsilon^*$.
    
\end{lemma}

\begin{proof}
From the definition we immediately get
\begin{align*}
    \widetilde{L}_{\varepsilon} g_{\varepsilon} = \Delta g_{\varepsilon} - \Lambda g_{\varepsilon} + \frac{n-3}{\abs{y}^2} \eta_{\varepsilon}'(\abs{y} \psi_{\varepsilon}(y)) g_{\varepsilon} = g_{\varepsilon}.
\end{align*}

Due to the decay of $\Lambda \psi_{\varepsilon}$, see \eqref{Decay of psi_e}, we get
\begin{align*}
    \abs{\partial^{\alpha}g_{\varepsilon}(y)} \lesssim \langle y \rangle^{-3-\abs{\alpha}}
\end{align*}

for every $\alpha \in \N_0^n$ and $y \in \R^n$. We therefore have $g_{\varepsilon} \in X_s^k$. That $g_{\varepsilon}$ actually belongs to the domain of $L_{\varepsilon}$ can be proven via a cutoff argument. We therefore obtain $g_{\varepsilon} \in \mc{D}(L_{\varepsilon}^X)$ with $L_{\varepsilon}^X \, g_{\varepsilon} = g_{\varepsilon}$.

The Lipschitz continuity follows from \eqref{Lipschitz continuity of psi_e} and Lemma 2.1 from \cite{DonSchWit25}, where we note that one actually obtains a quantitive estimate from this Lemma in the sense that if $f \in C^{\infty}(\R^n)$ satisfies estimates of the form
\begin{align*}
    |\partial^{\beta}f(x)| \lesssim A \langle x \rangle^{- \gamma - |\beta|}
\end{align*}
for some constant $A > 0$, we get $f \in \dot{H}^s(\R^n)$ for $s> \frac{n}{2} - \gamma$ with
\begin{align*}
    \norm{f}_{\dot{H}^s(\R^n)} \lesssim A.
\end{align*}

\end{proof}

We now use perturbative methods to exclude any other unstable spectral points of $L_{\varepsilon}^X$.   

\begin{proposition}\label{Proposition: Stability of spectrum of Le}
For any fixed $0 < \omega_0 < \widetilde{\omega}$, where $\widetilde{\omega}$ is the constant from Proposition \ref{Proposition: Properties of L_0}, there exists an $\varepsilon^{**}>0$ such that for every $\varepsilon \in \R$ with $\abs{\varepsilon} \leq \varepsilon^{**}$
\begin{align*}
    \sigma(L_{\varepsilon}^X) \subset \{ \lambda \in \C : \Re \lambda <  -\omega_0 \} \cup \{1\}.
\end{align*}

Moreover, $\lambda = 1$ is a simple eigenvalue of $L_{\varepsilon}^X$ with eigenspace spanned by $g_{\varepsilon}$. Furthermore, the associated Riesz projection
\begin{align}\label{Riesz projection}
    P_{\varepsilon} := \frac{1}{2\pi i} \int_{\partial B_{1\slash2}(1)} R_{L_{\varepsilon}^X}(\lambda) \, d\lambda
\end{align}

satisfies $\ran(P_{\varepsilon}) = \ker(I - L_{\varepsilon}^X) = \langle g_{\varepsilon} \rangle$. 
\end{proposition}

\begin{proof}
We first show that
\begin{align*}
    \sigma(L_{\varepsilon}^X) \cap \{\lambda \in \C : \Re \lambda \geq - \omega_0\} \subset B_{1\slash2}(1),
\end{align*}

so that any spectral point in the right half-plane $\{\lambda \in \C : \Re \lambda \geq - \omega_0\}$ must lie in a compact region around the isolated eigenvalue $\lambda = 1$.
Moreover, there exists a constant $C_{\omega_0}>0$ such that the resolvent satisfies 
\begin{align}\label{Resolvent estimate outside ball}
    \norm{R_{L_{\varepsilon}^X}(\lambda)} \leq C_{\omega_0}
\end{align}

for all $\lambda \in \closure{\mathbb{H}_{-\omega_0}}\, \backslash B_{1\slash2}(1)$ and all sufficiently small $\abs{\varepsilon}$.

To prove this we take $\lambda \in \closure{\mathbb{H}_{-\omega_0}}\, \backslash B_{1\slash2}(1)$ and $\abs{\varepsilon}\leq\varepsilon^*$. Since $\lambda$ then belongs to the resolvent set of $L_0^X$ (due to Proposition \ref{Proposition: Properties of L_0}) we obtain the following identity
\begin{align}\label{Identity for Neumann argument}
    \lambda - L_{\varepsilon}^X = \left(I-(L_{\varepsilon}^X-L_0^X)R_{L_0^X}(\lambda)\right)(\lambda-L_0^X).
\end{align}

Consequently, invertibility of $\lambda - L_{\varepsilon}^X$ is equivalent to the invertibility of $I-(L_{\varepsilon}^X-L_0^X)R_{L_0^X}(\lambda)$, which can be proven via a Neumann-series argument. For this, we decompose the resolvent into its stable and unstable components via the spectral projection $P_0$ from Proposition \ref{Proposition: Properties of L_0}
\begin{align}\label{Splitting of resolvent}
    R_{L_0^X}(\lambda) f = R_{L_0^X}(\lambda) (I - P_0) f + R_{L_0^X}(\lambda) P_0 f.
\end{align}

On the stable subspace, $\ran(I - P_0)$, the semigroup $\left(S_0^X(\tau)\right)_{\tau\geq0}$ decays exponentially, see \eqref{Decay on stable subspace}, so that we obtain for every $0 < \omega < \widetilde{\omega}$ 
\begin{align*}
    \norm{S_0^X(\tau)\left(I - P_0\right)u}_{s,k} \lesssim e^{- \omega \,\tau}\norm{\left(I - P_0\right)u}_{s,k}.
\end{align*}

Since the restriction of $\left(S_0^X(\tau)\right)_{\tau \geq 0}$ to $\ran(I - P_0)$ coincides with the semigroup generated by the restriction of $L_0^X$ to the stable subspace, we can apply \cite{EngNag00}, p.55, Theorem 1.10. This yields for every $0 < \omega < \widetilde{\omega}$ the existence of a constant $M_{\omega} \geq 1$ such that
\begin{align}\label{Resolvent estimate on stable subspace}
    \norm{R_{L_0^X}(\lambda)(I-P_0)} \leq \frac{M_{\omega}}{\Re \lambda + \omega}
\end{align}

for all $\lambda \in \mathbb{H}_{- \omega}$. By fixing $\omega = \frac{\omega_0 + \widetilde{\omega}}{2}$ in \eqref{Resolvent estimate on stable subspace} it follows that there exists a constant $C_{\omega_0} > 0$ such that 
\begin{align*}
    \norm{R_{L_0^X}(\lambda)(I-P_0)} \leq C_{\omega_0}
\end{align*}

for all $\lambda \in \closure{\mathbb{H}_{-\omega_0}}$. For the second term in Eq.~\eqref{Splitting of resolvent} we exploit the fact that $\lambda = 1$ is a simple eigenvalue with eigenspace spanned by $g_0$. Therefore, there exists a unique $h_0 \in X_s^k$ with
\begin{align*}
    R_{L_0^X}(\lambda) P_0 f = (\lambda-1)^{-1} \langle h_0, f \rangle_{s,k} \, g_0.
\end{align*}

for every $f \in X_s^k$. Combining \eqref{Splitting of resolvent} with the previous bounds yields a constant $C_{\omega_0}>0$ with
\begin{align*}
    \norm{R_{L_0^X}(\lambda)} \leq  C_{\omega_0}
\end{align*}

for every $\lambda \in \closure{\mathbb{H}_{-\omega_0}}\, \backslash B_{1\slash2}(1)$. From the Lipschitz-continuity of $L_{\varepsilon}^X$, see Proposition \ref{Proposition: Lipschitz continuity of full linear part}, we obtain the bound
\begin{align*}
    \norm{(L_{\varepsilon}^X-L_0^X)R_{L_0^X}(\lambda)} \lesssim |\varepsilon|
\end{align*}

for every $\lambda \in \closure{\mathbb{H}_{-\omega_0}}\, \backslash B_{1\slash2}(1)$ and every $\abs{\varepsilon}\leq\varepsilon^*$. If we now choose $\varepsilon^{**} \leq \varepsilon^*$ sufficiently small so that
\begin{align*}
    \norm{(L_{\varepsilon}^X-L_0^X)R_{L_0^X}(\lambda)} < 1
\end{align*}

the identity \eqref{Identity for Neumann argument} implies $\lambda \in \rho(L_{\varepsilon}^X)$. Moreover, we obtain the resolvent representation
\begin{align*}
    R_{L_{\varepsilon}^X}(\lambda) =  R_{L_0^X}(\lambda) \left(I-(L_{\varepsilon}^X-L_0^X)R_{L_0^X}(\lambda)\right)^{-1}
\end{align*}

from which \eqref{Resolvent estimate outside ball} follows.

To exclude any other unstable spectral points in the compact region $\overline{B_{1\slash2}(1)}$ we first notice that we have previously shown that $\partial B_{1\slash 2}(1)$ belongs to the resolvent set of $L_{\varepsilon}^X$. In particular, the Riesz projection $P_{\varepsilon}$ from \eqref{Riesz projection} is well-defined. Moreover, combining the resolvent identity with the Lipschitz-continuity of $L_{\varepsilon}^X$ and the uniform resolvent bounds on $\partial B_{1\slash 2}(1)$ we obtain for every $\lambda \in \partial B_{1\slash 2}(1)$
\begin{align*}
    \norm{R_{L_{\varepsilon}^X}(\lambda)-R_{L_{\kappa}^X}(\lambda)} \leq \norm{R_{L_{\varepsilon}^X}(\lambda)} \norm{L_{\varepsilon}^X-L_{\kappa}^X} \norm{R_{L_{\kappa}^X}(\lambda)} \lesssim \abs{\varepsilon-\kappa}.
\end{align*}

It follows that the resolvent $R_{L_{\varepsilon}^X}$ depends Lipschitz-continuously on the parameter $\varepsilon$ and hence, the same holds for the Riesz projection $P_{\varepsilon}$. By \cite{Kat95}, p.34, Lemma 4.10 we can conclude for every $\abs{\varepsilon} \leq \varepsilon^{**}$,
\begin{align*}
    \dim \ran (P_{\varepsilon}) = \dim \ran (P_0) = 1.
\end{align*}

Since we already have $\langle g_{\varepsilon} \rangle \subset \ker(I - L_{\varepsilon}^X) \subset \ran (P_{\varepsilon})$, it follows that all of these spaces coincide. Consequently, the only spectral point of $L_{\varepsilon}^X$ in $B_{1\slash2}(1)$ is the simple eigenvalue $\lambda = 1$.

\end{proof}

With these results we can now also prove the exponential decay/growth of the semigroup on the stable/unstable subspace, respectively.  

\begin{proposition}\label{Projection properties}
Fix an arbitrary $ 0 < \omega_0 < \widetilde{\omega}$, where $\widetilde{\omega}$ is the constant from Proposition \ref{Proposition: Properties of L_0}. Furthermore, let $\abs{\varepsilon} \leq \varepsilon^{**}$, where $\varepsilon^{**}$ is the constant from Proposition \ref{Proposition: Stability of spectrum of Le}. Then the projection $P_{\varepsilon}$ commutes with the semigroup $S_{\varepsilon}^X$. In particular, we have 
\begin{align}\label{Exponential growth on unstable subspace}
    P_{\varepsilon}\, S_{\varepsilon}^X(\tau) = S_{\varepsilon}^X(\tau) P_{\varepsilon} = e^{\tau}\,P_{\varepsilon} 
\end{align}

for all $\tau \geq 0$. For the stable subspace we have
\begin{align}\label{Exponential decay on stable subspace}
    \norm{S_{\varepsilon}^X(\tau) (I-P_{\varepsilon}) u}_{s,k} \lesssim e^{-\omega_0 \, \tau} \norm{(I-P_{\varepsilon})u}_{s,k}
\end{align}

as well as
\begin{align}\label{Lipschitz property of semigroup on stable subspace}
    \norm{[S_{\varepsilon}^X(\tau) (I-P_{\varepsilon}) - S_{\kappa}^X(\tau) (I-P_{\kappa})]u}_{s,k} \lesssim e^{-\omega_0 \, \tau} \abs{\varepsilon-\kappa} \norm{u}_{s,k} 
\end{align}

for all $u\in X_s^k, \tau\geq 0$ and $ \abs{\varepsilon},\abs{\kappa}\leq \varepsilon^{**}$.
\end{proposition}

\begin{proof}
Because the semigroup $S_{\varepsilon}^X(\tau)$ commutes with its generator $L_{\varepsilon}^X$ it also commutes with the associated resolvent $R_{L_{\varepsilon}^X}$. Consequently, it commutes with the spectral projection $P_{\varepsilon}$ as well. Eq.~\eqref{Exponential growth on unstable subspace} then follows from uniqueness of solutions to $\partial_{\tau}u(\tau) = L_{\varepsilon}^Xu(\tau)$.

The decay on the stable subspace \eqref{Exponential decay on stable subspace} follows from \cite{Ost23} Theorem A.1 once there exists a constant $M_{\omega_0}>0$ with
\begin{align}\label{Uniform bound of resolvent on stable subspace}
    \norm{R_{L_{\varepsilon}^X}(\lambda)\left(I-P_{\varepsilon}\right)} \leq M_{\omega_0}
\end{align}

for all $\lambda \in \closure{\mathbb{H}_{-\omega_0}}$ and $\abs{\varepsilon} \leq \varepsilon^{**}$ for a sufficiently small $\varepsilon^{**} > 0$.

From Proposition \ref{Proposition: Stability of spectrum of Le} we already know that there exists a constant $M_{\omega_0}$ such that
\begin{align*}
    \norm{R_{L_{\varepsilon}^X}(\lambda)} \leq M_{\omega_0}
\end{align*}

holds for every $\lambda \in \closure{\mathbb{H}_{-\omega_0}}\, \backslash B_{1\slash2}(1)$ and $\abs{\varepsilon} \leq \varepsilon^{**}$. Since $R_{L_{\varepsilon}^X}(\lambda)\left(I-P_{\varepsilon}\right)$ is analytic in $\mathbb{H}_{-\omega_0}$ and coincides with the resolvent of $L_{\varepsilon}^X$ restricted to the range of $I-P_{\varepsilon}$, the estimate \eqref{Uniform bound of resolvent on stable subspace} extends to all $\lambda \in \closure{B_{1\slash2}(1)}$ as well.

To prove the Lipschitz estimate of the semigroup on the stable subspace \eqref{Lipschitz property of semigroup on stable subspace} let $u \in \mc{D}(L^X) = \mc{D}(L_{\varepsilon}^X) = \mc{D}(L_{\kappa}^X)$, choose $\omega_0 < \omega_1 < \widetilde{\omega}$, fix $\tau > 0$ and define
\begin{align*}
    F(\sigma) := S_{\varepsilon}^X(\tau-\sigma) S_{\kappa}^X(\sigma) u \quad \text{for} \quad 0 \leq \sigma \leq \tau.
\end{align*}
We obtain
\begin{align*}
    F'(\sigma) & = - S_{\varepsilon}^X(\tau-\sigma)L_{\varepsilon}^XS_{\kappa}^X(\sigma)u + S_{\varepsilon}^X(\tau-\sigma)L_{\kappa}^XS_{\kappa}^X(\sigma)u \\ & = - S_{\varepsilon}^X(\tau-\sigma)(L_{\varepsilon}^X-L_{\kappa}^X)S_{\kappa}^X(\sigma)u
\end{align*}
and integrating over $[0,\tau]$ gives
\begin{align}\label{Difference of semigroups}
    S_{\varepsilon}^X(\tau)u - S_{\kappa}^X(\tau)u = \int_0^{\tau} S_{\varepsilon}^X(\tau-\sigma)(L_{\varepsilon}^X-L_{\kappa}^X)S_{\kappa}^X(\sigma)u \, d\sigma.
\end{align}
We now apply this to $(I-P_{\kappa})u$ instead of just $u$ and obtain
\begin{align*}
    S_{\kappa}^X(\tau) (I-P_{\kappa})u = S_{\varepsilon}^X(\tau) (I-P_{\kappa})u - \int_0^{\tau} S_{\varepsilon}^X(\tau-\sigma)(L_{\varepsilon}^X-L_{\kappa}^X)S_{\kappa}^X(\sigma) (I-P_{\kappa})u \, d\sigma.
\end{align*}
If we now apply $P_{\varepsilon}$ to this, we obtain with \eqref{Exponential growth on unstable subspace}
\begin{align*}
    P_{\varepsilon} S_{\kappa}^X(\tau) (I - P_{\kappa}) u = e^{\tau} P_{\varepsilon} (I - P_{\kappa}) u - \int_0^{\tau} e^{\tau-\sigma}P_{\varepsilon} (L_{\varepsilon}^X-L_{\kappa}^X)S_{\kappa}^X(\sigma)(I-P_{\kappa}) u \, d\sigma.
\end{align*}
Multiplying by $e^{-\tau}$ we get
\begin{align*}
    e^{-\tau}P_{\varepsilon} S_{\kappa}^X(\tau) (I - P_{\kappa}) u =  P_{\varepsilon} (I - P_{\kappa}) u - \int_0^{\tau} e^{-\sigma}P_{\varepsilon} (L_{\varepsilon}^X-L_{\kappa}^X)S_{\kappa}^X(\sigma)(I-P_{\kappa}) u \, d\sigma
\end{align*}
and from \eqref{Exponential decay on stable subspace} we get that the left-hand side converges towards 0 as $\tau \to \infty$ and that the integral on the right-hand side converges. We therefore have
\begin{align}\label{cancellation identity}
    P_{\varepsilon} (I - P_{\kappa}) u = \int_0^{\infty} e^{-\sigma}P_{\varepsilon} (L_{\varepsilon}^X-L_{\kappa}^X)S_{\kappa}^X(\sigma)(I-P_{\kappa}) u \, d\sigma.
\end{align}
We now obtain with \eqref{Difference of semigroups}
\begin{align*}
    S_{\varepsilon}^X(\tau) (I-P_{\varepsilon}) u & - S_{\kappa}^X(\tau) (I-P_{\kappa}) u  = S_{\varepsilon}^X(\tau) (P_{\kappa}-P_{\varepsilon}) u + [S_{\varepsilon}^X(\tau) - S_{\kappa}^X(\tau)] (I-P_{\kappa}) u \\ & = S_{\varepsilon}^X(\tau) (P_{\kappa}-P_{\varepsilon}) u + \int_0^{\tau} S_{\varepsilon}^X(\tau-\sigma)(L_{\varepsilon}^X-L_{\kappa}^X)S_{\kappa}^X(\sigma)(I-P_{\kappa}) u \, d\sigma.
\end{align*}
If we now use $S_{\varepsilon}^X(\tau) = S_{\varepsilon}^X(\tau)(I-P_{\varepsilon}) + e^{\tau}P_{\varepsilon}$ and \eqref{cancellation identity} we obtain the following representation formula
\begin{align*}
    S_{\varepsilon}^X(\tau) & (I-P_{\varepsilon}) u  - S_{\kappa}^X(\tau) (I-P_{\kappa}) u \\ = &  S_{\varepsilon}^X(\tau) (I-P_{\varepsilon}) (P_{\kappa}-P_{\varepsilon}) u + \int_0^{\tau} S_{\varepsilon}^X(\tau-\sigma)(I-P_{\varepsilon})(L_{\varepsilon}^X-L_{\kappa}^X)S_{\kappa}^X(\sigma)(I-P_{\kappa}) u \, d\sigma \\ & + e^{\tau}\left[P_{\varepsilon}(P_{\kappa}-P_{\varepsilon}) u + \int_0^{\tau} e^{-\sigma}P_{\varepsilon}(L_{\varepsilon}^X-L_{\kappa}^X)S_{\kappa}^X(\sigma)(I-P_{\kappa}) u \, d\sigma\right] \\ = & S_{\varepsilon}^X(\tau) (I-P_{\varepsilon}) (P_{\kappa}-P_{\varepsilon}) u + \int_0^{\tau} S_{\varepsilon}^X(\tau-\sigma)(I-P_{\varepsilon})(L_{\varepsilon}^X-L_{\kappa}^X)S_{\kappa}^X(\sigma)(I-P_{\kappa}) u \, d\sigma \\ & + e^{\tau}\left[P_{\varepsilon}(P_{\kappa}-I) u + \int_0^{\tau} e^{-\sigma}P_{\varepsilon}(L_{\varepsilon}^X-L_{\kappa}^X)S_{\kappa}^X(\sigma)(I-P_{\kappa}) u \, d\sigma\right] \\ = & S_{\varepsilon}^X(\tau) (I-P_{\varepsilon}) (P_{\kappa}-P_{\varepsilon}) u + \int_0^{\tau} S_{\varepsilon}^X(\tau-\sigma)(I-P_{\varepsilon})(L_{\varepsilon}^X-L_{\kappa}^X)S_{\kappa}^X(\sigma)(I-P_{\kappa}) u \, d\sigma \\ & - e^{\tau} \int_{\tau}^{\infty} e^{-\sigma} P_{\varepsilon} (L_{\varepsilon}^X-L_{\kappa}^X)S_{\kappa}^X(\sigma)(I-P_{\kappa}) u \, d\sigma.
\end{align*}

For the first term we obtain with \eqref{Exponential decay on stable subspace} and the Lipschitz continuity of $P_{\varepsilon}$ for every $u \in \mc{D}(L^X)$
\begin{align*}
    \norm{S_{\varepsilon}^X(\tau) (I-P_{\varepsilon}) (P_{\kappa}-P_{\varepsilon})u}_{s,k} \lesssim e^{-\omega_0\tau} |\varepsilon-\kappa| \norm{u}_{s,k}.
\end{align*}
For the second term we use the Lipschitz continuity of $L_{\varepsilon}^X$ and get
\begin{align*}
    & \norm{\int_0^{\tau} S_{\varepsilon}^X(\tau-\sigma)(I-P_{\varepsilon})(L_{\varepsilon}^X-L_{\kappa}^X)S_{\kappa}^X(\sigma)(I-P_{\kappa}) u \, d\sigma}_{s,k} \\  \lesssim |\varepsilon-\kappa| & \int_0^{\tau} e^{-\omega_1(\tau-\sigma)}e^{-\omega_1\sigma} \, d\sigma \norm{u}_{s,k} = |\varepsilon-\kappa|\tau e^{-\omega_1\tau} \norm{u}_{s,k} \lesssim |\varepsilon-\kappa| e^{-\omega_0\tau} \norm{u}_{s,k}.
\end{align*}
For the last term we obtain analogously
\begin{align*}
    e^\tau \norm{\int_{\tau}^{\infty} e^{-\sigma} P_{\varepsilon} (L_{\varepsilon}^X-L_{\kappa}^X)S_{\kappa}^X(\sigma)(I-P_{\kappa}) u \, d\sigma}_{s,k} & \lesssim |\varepsilon-\kappa|e^{\tau} \int_{\tau}^{\infty} e^{-(1+\omega_0)\sigma} \,d\sigma \norm{u}_{s,k} \\ & \lesssim |\varepsilon-\kappa| e^{-\omega_0\tau} \norm{u}_{s,k},
\end{align*}
so that the claim now follows from the density of $\mc{D}(L^X)$ in $X_s^k$.

\end{proof}

\subsection{The full nonlinear Cauchy problem for the perturbation}\label{Section: Nonlinear Cauchy Problem}

We now turn to the full nonlinear Cauchy problem \eqref{Cauchy problem for perturbation stability analysis} but for now with arbitrary small initial data $u \in X_s^k$ that is
\begin{align}\label{Cauchy problem for arbitrary initial data}
    \begin{cases}
        \partial_{\tau} \phi_{\varepsilon}(\tau) & = \widetilde{L}_{\varepsilon}(\phi_{\varepsilon}(\tau)) + \widehat{N}_{\varepsilon} (\phi_{\varepsilon}(\tau)), \\
        \phi_{\varepsilon}(0) & = u,\\
    \end{cases}
\end{align}

By Duhamel’s formula, this Cauchy problem can be written as the following integral equation
\begin{align}\label{Integral equation for perturbation}
    \phi_{\varepsilon}(\tau) = S_{\varepsilon}^X(\tau)u + \int_0^{\tau} S_{\varepsilon}^X(\tau-\tau') \widehat{N}_{\varepsilon}(\phi_{\varepsilon}(\tau'))\, d\tau'.
\end{align}

We fix $\omega := \widetilde{\omega}\slash 2$, where $\widetilde{\omega}$ is the constant from Proposition \ref{Proposition: Properties of L_0}, and set $\overline{\varepsilon} := \varepsilon^{**}(\omega)$, with $\varepsilon^{**}(\omega)$ as in Proposition \ref{Proposition: Stability of spectrum of Le}. We now first prove a local Lipschitz estimate for the nonlinearity $\widehat{N}_{\varepsilon}$ defined in \eqref{Definition: Nonlinearity in stability analysis}.

\begin{lemma}\label{Local Lipschitz continuity for nonlinearity II}
For every $\abs{\varepsilon} \leq \overline{\varepsilon}$ we have $\widehat{N}_{\varepsilon} : X_s^k \to X_s^k$ as well as
\begin{align*}
    \norm{\widehat{N}_{\varepsilon}(u)-\widehat{N}_{\kappa}(v)}_{s,k} \lesssim \left(\norm{u}_{s,k} + \norm{v}_{s,k}\right)\norm{u-v}_{s,k} + \left(\norm{u}_{s,k}^2 + \norm{v}_{s,k}^2\right) \abs{\varepsilon-\kappa}
\end{align*}

for all $\abs{\varepsilon}, \abs{\kappa} \leq \overline{\varepsilon}$ and all $u,v \in \mc{B}_{\delta} \subset X_s^k$ for $0 < \delta \leq 1$ fixed.
\end{lemma}

\begin{proof}
The proof is analogous to the proof of Lemma \ref{Estimate for nonlinearity} using the Schauder estimate from Lemma \ref{Lemma: Schauder}. 
    
\end{proof}

The stability of $\psi_{\varepsilon}$ would now follow if we could show global existence and exponential decay of functions satisfying \eqref{Integral equation for perturbation} for arbitrary small initial data $u \in X_s^k$. Unfortunately, the unstable mode at $\lambda = 1$ of the linearized operator $L_{\varepsilon}^X$ prevents us from concluding this immediately. To overcome this we follow the standard approach (see for example Section 4.2 in \cite{DonSchWit25}) by subtracting a correction term
\begin{align*}
    C(\phi,\varepsilon,u) := P_{\varepsilon} \left( u + \int_0^{\infty}e^{-\tau'}\, \widehat{N}_{\varepsilon}(\phi(\tau'))\,d\tau'\right)
\end{align*}

to the initial data, which then stabilizes the evolution. More concretely, we introduce the Banach space
\begin{align*}
    \mc{X} := \{\phi \in C([0,\infty), X_s^k) : \norm{\phi}_{\mc{X}} := \sup_{\tau>0} e^{\omega\tau} \norm{\phi(\tau)}_{s,k} < \infty\}
\end{align*}

and want to find a $\phi \in \mc{X}$ solving the following integral equation
\begin{align}\label{Fixed point for perturbation}
    \phi(\tau) = S_{\varepsilon}^X(\tau)\left[u - C(\phi, \varepsilon, u)\right] + \int_0^{\tau} S_{\varepsilon}^X(\tau - \tau') \widehat{N}_{\varepsilon}(\phi(\tau'))\,d\tau'.
\end{align}

Since this is a fixed-point equation we define
\begin{align*}
    K(\phi,\varepsilon,u)(\tau) := S_{\varepsilon}^X(\tau)\left[u - C(\phi, \varepsilon, u)\right] + \int_0^{\tau} S_{\varepsilon}^X(\tau - \tau') \widehat{N}_{\varepsilon}(\phi(\tau'))\,d\tau'
\end{align*}

and want to show that $K(\cdot, \varepsilon, u)$ is a well-defined contraction on $\mc{X}_{\delta} := \{\phi \in \mc{X} : ~ \norm{\phi}_{\mc{X}} \leq \delta\}$ for sufficiently small $\delta > 0$ and fixed $(\varepsilon, u)$.   

\begin{proposition}\label{Proposition: Cauchy problem with modified initial data}
There exist constants $0 < \delta_0 < 1$ and $C_0 > 1$ such that for all $0 < \delta \leq \delta_0$ and $C \geq C_0$ there exists for every $\abs{\varepsilon} \leq \overline{\varepsilon}$ and every $u \in X_s^k$ with $\norm{u}_{s,k} \leq \frac{\delta}{C}$ a unique function $\phi_{\varepsilon}(u) \in \mc{X}_{\delta}$ such that \eqref{Fixed point for perturbation} holds for all $\tau \geq 0$.

Additionally, the map $(u,\varepsilon) \mapsto \phi_{\varepsilon}(u)$ is Lipschitz continuous in the sense that the following estimate holds for all $u, v \in X_s^k$ with $\norm{u}_{s,k}, \norm{v}_{s,k} \leq \frac{\delta}{C}$ and all $\abs{\varepsilon}, \abs{\kappa} \leq \overline{\varepsilon}$
\begin{align*}
    \norm{\phi_{\varepsilon}(u) - \phi_{\kappa}(v)}_{\mc{X}} \lesssim \norm{u-v}_{s,k} + \abs{\varepsilon-\kappa}.
\end{align*}

\end{proposition}

\begin{proof}
First, we show that the map $K_{(u,\varepsilon)}(\phi) := K(\phi,\varepsilon,u)$ is a well-defined contraction on $\mc{X}_{\delta}$ for fixed $|\varepsilon| \leq \overline{\varepsilon}$ all sufficiently large  $C>1$, sufficiently small $\delta >0$ and all $u \in X_s^k$ with $\norm{u}_{s,k}\leq\frac{\delta}{C}$. To see this we take $\phi \in \mc{X}_{\delta}$ and $\tau \geq 0$ and write $K_{(u,\varepsilon)}$ in the following way
\begin{align}
    &K_{(u,\varepsilon)}(\phi)(\tau)\nonumber\\ = &\,S_{\varepsilon}^X(\tau)(I-P_{\varepsilon})u - \int_{\tau}^{\infty}e^{\tau-\tau'}\,P_{\varepsilon}\,\widehat{N}_{\varepsilon}(\phi(\tau'))\,d\tau' + \int_0^{\tau}S_{\varepsilon}^X(\tau-\tau')(I-P_{\varepsilon}) \widehat{N}_{\varepsilon}(\phi(\tau'))\,d\tau'.\label{Rewritten fixed-point operator} 
\end{align}

From this we obtain with Proposition \ref{Projection properties} and Lemma \ref{Local Lipschitz continuity for nonlinearity II}
\begin{align*}
    \norm{K_{(u,\varepsilon)}(\phi)(\tau)}_{s,k} \lesssim \frac{\delta}{C}e^{-\omega\,\tau} + \delta^2 e^{-2\omega\,\tau} + \delta^2 e^{-\omega\,\tau} (e^{-\omega\,\tau} + 1) \lesssim \left(\frac{1}{C}+\delta\right) \delta e^{-\omega\,\tau}
\end{align*}

so that we have
\begin{align*}
    \norm{K_{(u,\varepsilon)}(\phi)(\tau)}_{s,k} \leq \delta e^{-\omega\,\tau}
\end{align*}

if we choose $C \geq C_0$ and $0 < \delta \leq \delta_0$ with $C_0$ sufficiently large and $\delta_0>0$ sufficiently small. 

Since the continuity of the mapping $\tau \mapsto K_{(u,\varepsilon)}(\phi)(\tau)$ follows from the dominated convergence we conclude that $K_{(u,\varepsilon)}: \mc{X}_{\delta} \to  \mc{X}_{\delta}$ is well-defined. 

To show that $K_{(u,\varepsilon)}$ is a contraction on $\mc{X}_{\delta}$ (for potentially even smaller $\delta>0$) we take $\phi,\psi \in \mc{X}_{\delta}$ and calculate for every $\tau \geq 0$ by using again the representation of $K_{(u,\varepsilon)}$ from Eq.~\eqref{Rewritten fixed-point operator}
\begin{align*}
    & \norm{K_{(u,\varepsilon)}(\phi)(\tau)-K_{(u,\varepsilon)}(\psi)(\tau)}_{s,k}\\ \lesssim &\,\int_{\tau}^{\infty}e^{\tau-\tau'}\lVert\widehat{N}_{\varepsilon}(\phi(\tau'))-\widehat{N}_{\varepsilon}(\psi(\tau'))\rVert_{s,k}\,d\tau' + \int_0^{\tau}e^{-\omega\,(\tau-\tau')}\lVert\widehat{N}_{\varepsilon}(\phi(\tau'))-\widehat{N}_{\varepsilon}(\psi(\tau'))\rVert_{s,k}\,d\tau'\\ \lesssim & \, \left(\delta e^{-2\omega\,\tau}+\delta e^{-\omega\,\tau}(e^{-\omega\,\tau}+1)\right)\norm{\phi-\psi}_{\mc{X}}.
\end{align*}

If we now choose $\delta_0>0$ sufficiently small we get
\begin{align*}
    \norm{K_{(u,\varepsilon)}(\phi)-K_{(u,\varepsilon)}(\psi)}_{\mc{X}} \leq \frac{1}{2}\norm{\phi-\psi}_{\mc{X}},
\end{align*}

for all $0 < \delta \leq \delta_0$, $C \geq C_0$ and all $\varepsilon \in \R$ with $\abs{\varepsilon} \leq \overline{\varepsilon}$ so that the existence of a unique $\phi_{\varepsilon}(u) \in \mc{X}_{\delta}$ follows from the contraction mapping principle. 

It remains to prove that the map $(u,\varepsilon) \mapsto \phi_{\varepsilon}(u) \in \mc{X}_{\delta}$ is Lipschitz continuous. For this we take $(u,\varepsilon),(v,\kappa) \in \mc{B}_{\frac{\delta}{C}}\times [-\overline{\varepsilon},\overline{\varepsilon}]$ and obtain by the previous considerations functions $\phi_{\varepsilon}(u),\phi_{\kappa}(v) \in \mc{X}_{\delta}$ solving
\begin{align*}
     \phi_{\varepsilon}(u)(\tau) = K(\phi_{\varepsilon}(u),\varepsilon,u)(\tau) \quad \text{and} \quad \phi_{\kappa}(v)(\tau) = K(\phi_{\kappa}(v),\kappa,v)(\tau) \quad \forall \, \tau \geq 0. 
\end{align*}

We now show that
\begin{align*}
    \norm{K(\phi_{\varepsilon}(u),\varepsilon,u)-K(\phi_{\kappa}(v),\kappa,v)}_{\mc{X}} \lesssim \norm{u-v}_{s,k} + \abs{\varepsilon-\kappa}.
\end{align*}
For this we take $\tau \geq 0$ and estimate the terms in \eqref{Rewritten fixed-point operator} separately. For the first term we simply get from \eqref{Lipschitz property of semigroup on stable subspace}
\begin{align*}
    \norm{S_{\varepsilon}^X(\tau)(I-P_{\varepsilon})u - S_{\kappa}^X(\tau)(I-P_{\kappa})v} \lesssim \frac{\delta}{C}e^{-\omega\,\tau}\abs{\varepsilon-\kappa} + e^{-\omega\,\tau}\norm{u-v}_{s,k}.
\end{align*}

For the second term we apply Lemma \ref{Local Lipschitz continuity for nonlinearity II} and the Lipschitz continuity of $P_{\varepsilon}$ to get
\begin{align*}
    &\int_{\tau}^{\infty}e^{\tau-\tau'}\lVert P_{\varepsilon}\,\widehat{N}_{\varepsilon}(\phi_{\varepsilon}(u)(\tau'))-P_{\kappa}\,\widehat{N}_{\kappa}(\phi_{\kappa}(v)(\tau'))\rVert_{s,k}\,d\tau'\\ \lesssim & \, \abs{\varepsilon-\kappa}\int_{\tau}^{\infty}e^{\tau-\tau'} \lVert\widehat{N}_{\kappa}(\phi_{\kappa}(v)(\tau'))\rVert_{s,k}\,d\tau' + \int_{\tau}^{\infty}e^{\tau-\tau'}\lVert\widehat{N}_{\varepsilon}(\phi_{\varepsilon}(u)(\tau'))-\widehat{N}_{\kappa}(\phi_{\kappa}(v)(\tau'))\rVert_{s,k}\,d\tau' \\ \lesssim & \, \abs{\varepsilon-\kappa}\, \delta^2 \, e^{-2\omega\,\tau} + \delta \, e^{-2\omega\,\tau} \norm{\phi_{\varepsilon}(u)-\phi_{\kappa}(v)}_{\mc{X}}
\end{align*}

and for the last term we similarly get
\begin{align*}
    &\int_0^{\tau}\norm{S_{\varepsilon}^X(\tau-\tau')(I-P_{\varepsilon}) \widehat{N}_{\varepsilon}(\phi_{\varepsilon}(u)(\tau')) - S_{\kappa}^X(\tau-\tau')(I-P_{\kappa}) \widehat{N}_{\kappa}(\phi_{\kappa}(v)(\tau'))}_{s,k}\,d\tau' \\ \lesssim & \,\abs{\varepsilon-\kappa} \int_0^{\tau}e^{-\omega\,(\tau-\tau')}\lVert\widehat{N}_{\kappa}(\phi_{\kappa}(v)(\tau'))\rVert_{s,k}\,d\tau'\\ & + \int_0^{\tau}e^{-\omega\,(\tau-\tau')}\lVert\widehat{N}_{\varepsilon}(\phi_{\varepsilon}(u)(\tau')) - \widehat{N}_{\kappa}(\phi_{\kappa}(v)(\tau'))\rVert_{s,k}\,d\tau' \\ \lesssim & \, \abs{\varepsilon-\kappa} \, \delta^2 \, e^{-\omega\,\tau} + \delta \, e^{-\omega\,\tau} \norm{\phi_{\varepsilon}(u) - \phi_{\kappa}(v)}_{\mc{X}}.
\end{align*}

With that we obtain 
\begin{align*}
    \norm{\phi_{\varepsilon}(u)-\phi_{\kappa}(v)}_{\mc{X}} =& \norm{K(\phi_{\varepsilon}(u),\varepsilon,u)-K(\phi_{\kappa}(v),\kappa,v)}_{\mc{X}} \\ \lesssim & \, \abs{\varepsilon-\kappa} + \norm{u-v}_{s,k} + \delta \norm{\phi_{\varepsilon}(u)-\phi_{\kappa}(v)}_{\mc{X}}.
\end{align*}

For small enough $\delta>0$ we get $\norm{\phi_{\varepsilon}(u)-\phi_{\kappa}(v)}_{\mc{X}} \lesssim \abs{\varepsilon-\kappa} + \norm{u-v}_{s,k}$ as desired.
    
\end{proof}

To now obtain a solution to the original integral equation \eqref{Integral equation for perturbation} we adjust the blowup time $T$ so that the corresponding correction term vanishes for this particular choice of $T$.

We therefore return to the Cauchy problem \eqref{Cauchy problem for perturbation stability analysis} but now with the specific initial data and define 
\begin{align*}
U_{\varepsilon}(v,T) := v^T + \psi_{\varepsilon}^T - \psi_{\varepsilon} = \sqrt{T}v(\sqrt{T}\cdot) + \sqrt{T}\psi_{\varepsilon}(\sqrt{T}\cdot) - \psi_{\varepsilon}.
\end{align*}

\begin{lemma}\label{Lemma: Initial data operator}
Take $0 < \delta \leq \frac{1}{2}$. For every $\varepsilon \in \R$ with $\abs{\varepsilon} \leq \overline{\varepsilon}$ and every fixed $v \in X_s^k$ the map
\begin{align*}
    U_{\varepsilon}(v,\cdot) : [1-\delta,1+\delta] \to  X_s^k, ~ T \mapsto U_{\varepsilon}(v,T)
\end{align*}

is continuous and for every $T \in [1 - \delta, 1 + \delta]$ the initial data operator can be written as
\begin{align}\label{Equation for initial data operator}
    U_{\varepsilon}(v,T) = v^T + (T-1) g_{\varepsilon} + R_{\varepsilon}(T),
\end{align}

where $R_{\varepsilon}$ is a remainder term satisfying 
\begin{align*}
    \norm{R_{\varepsilon}(T)}_{s,k} \leq M_{\varepsilon}|T-1|^2
\end{align*}

for a constant $M_{\varepsilon} > 0$.
\end{lemma}

\begin{proof}
The continuity of $U_{\varepsilon}(v, \cdot)$ follows along the lines as for example ... . Eq.~\eqref{Equation for initial data operator} follows from a Taylor expansion applied to the map $[1 - \delta, 1 + \delta] \to X_s^k, ~ T \mapsto \psi^T_{\varepsilon}$ using the following fact
\begin{align*}
    \partial_T \restr{\sqrt{T}\psi_{\varepsilon}(\sqrt{T}\cdot)}{T=1} = \Lambda \psi_{\varepsilon} = g_{\varepsilon}. 
\end{align*}

The remainder term $R_{\varepsilon}(T)$ then satisfies 
\begin{align*}
    \norm{R_{\varepsilon}(T)}_{s,k} \lesssim (T-1)^2 \sum_{j=0}^2 \norm{\Lambda^j\psi_{\varepsilon}}_{s,k}
\end{align*}

and we remark that due to the decay of $g_{\varepsilon}$ we have that also $\Lambda g_{\varepsilon}$ belongs to $X_s^k$ so that the right-hand side of the above inequality is finite.

\end{proof}

Now, we are in the position to prove the central result of this section.

\begin{theorem}\label{Theorem: Corrected blowup time}
For any $\varepsilon \in \R$ with $\abs{\varepsilon} \leq \overline{\varepsilon}$, there are constants $0 < \delta_{\varepsilon} < 1$ and $C_{\varepsilon} > 1$ such that for all $0 < \delta \leq \delta_{\varepsilon}$ and all $C \geq C_{\varepsilon}$ the following statement holds: For every $v \in X_s^k$ with $\norm{v}_{s,k} \leq \frac{\delta}{C^2}$ there exists a $T_{\varepsilon} = T_{\varepsilon}(v) \in [1-\frac{\delta}{C}, 1+\frac{\delta}{C}]$ and a unique $\phi_{\varepsilon} \in C([0,\infty);  X_s^k)$ satisfying
\begin{align}\label{Integral equation with modification}
    \phi_{\varepsilon}(\tau) = S_{\varepsilon}^X(\tau) U_{\varepsilon}(v,T_{\varepsilon})+\int_0^{\tau}S_{\varepsilon}^X(\tau-\tau')\widehat{N}_{\varepsilon}(\phi_{\varepsilon}(\tau'))\,d\tau' \quad \text{for all} \quad \tau \geq 0.
\end{align}

Furthermore,
\begin{align*}
    \norm{\phi_{\varepsilon}(\tau)}_{s,k} \leq \delta e^{- \omega \tau}, \quad \forall \tau \geq 0.
\end{align*}  
\end{theorem}

\begin{proof}
Let $0 < \delta \leq \delta_0$ and $C \geq C_0 \geq 1$ with $\delta_0$ and $C_0$ as in Proposition \ref{Proposition: Cauchy problem with modified initial data}. Let  $\abs{\varepsilon} \leq \overline{\varepsilon}$ and $v \in X_s^k$ with $\norm{v}_{s,k} \leq \frac{\delta}{C^2}$ . Then  we obtain from Lemma \ref{Lemma: Initial data operator} 
\begin{align*}
    \norm{U_{\varepsilon}(v,T)}_{s,k} \lesssim &\norm{v^T}_{s,k} + \abs{T-1}\norm{g_{\varepsilon}}_{s,k} + \norm{R_{\varepsilon}(T)}_{s,k}\\ \lesssim &\frac{\delta}{C^2} + \frac{\delta}{C} A_{\varepsilon} + \frac{\delta^2}{C^2} M_{\varepsilon}
\end{align*}

for every $T \in [1-\frac{\delta}{C}, 1+\frac{\delta}{C}]$ where $A_{\varepsilon}, M_{\varepsilon} > 0$ are some constants depending on $\varepsilon$. If we now choose $\delta$ sufficiently small and $C$ sufficiently large we obtain for every $T \in [1-\frac{\delta}{C}, 1+\frac{\delta}{C}]$  from Proposition \ref{Proposition: Cauchy problem with modified initial data} the existence of a unique $\phi_{\varepsilon} = \phi_{\varepsilon}(v,T) \in \mc{X}_{\delta}$ which solves
\begin{align*}
    \phi_{\varepsilon}(\tau) = S_{\varepsilon}^X(\tau)\left[U_{\varepsilon}(v,T)-C(\phi_{\varepsilon},\varepsilon,U_{\varepsilon}(v,T))\right] + \int_0^{\tau} S_{\varepsilon}^X(\tau-\tau')\widehat{N}_{\varepsilon}(\phi_{\varepsilon}(\tau'))\,d\tau'.
\end{align*}

Since $C$ takes values in $\ran P_{\varepsilon} = \langle g_{\varepsilon} \rangle$ it is enough to show,  given $v$, the existence of a $T$ such that
\begin{align}\label{Correction term orthogonal to unstable subspace}
    \langle C(\phi_{\varepsilon}(v,T),\varepsilon,U_{\varepsilon}(v,T)), g_{\varepsilon}\rangle_{s,k} = 0.
\end{align}

Due to Lemma \ref{Lemma: Initial data operator} and the definition of $C$ this equation reads as
\begin{align*}
    0 = \langle P_{\varepsilon} v^T, g_{\varepsilon} \rangle_{s,k} + (T-1) \norm{g_{\varepsilon}}^2_{s,k} + \langle P_{\varepsilon} R_{\varepsilon}(T), g_{\varepsilon} \rangle_{s,k} +  \langle P_{\varepsilon} \int_0^{\infty} e^{-\tau'} \widehat{N}_{\varepsilon}(\phi_{\varepsilon}(\tau'))\,d\tau', g_{\varepsilon} \rangle_{s,k},
\end{align*}

which can be written as a fixed-point equation for $T \in [1-\frac{\delta}{C}, 1+\frac{\delta}{C}]$,
\begin{align}\label{Fixed-point equation for T}
    T = 1 - \langle P_{\varepsilon} v^T, \widehat{g}_{\varepsilon} \rangle_{s,k} - \langle P_{\varepsilon}  R_{\varepsilon}(T), \widehat{g}_{\varepsilon} \rangle_{s,k} -  \langle P_{\varepsilon} \int_0^{\infty} e^{-\tau'} \widehat{N}_{\varepsilon}(\phi_{\varepsilon}(\tau'))\,d\tau', \widehat{g}_{\varepsilon} \rangle_{s,k}, 
\end{align}

where we have set $\widehat{g}_{\varepsilon} = g_{\varepsilon} / \norm{g_{\varepsilon}}^2_{s,k}$. Now we obtain from the assumptions on $v$, the fact that $\phi_{\varepsilon}$ belongs to $\mc{X}_{\delta}$ and Lemma \ref{Lemma: Initial data operator} as well as Lemma \ref{Local Lipschitz continuity for nonlinearity II} the following estimate
\begin{align*}
    &\abs{\langle P_{\varepsilon} v^T, \widehat{g}_{\varepsilon} \rangle_{s,k}} + \abs{\langle P_{\varepsilon} R_{\varepsilon}(T), \widehat{g}_{\varepsilon} \rangle_{s,k}} +  \abs{\langle P_{\varepsilon} \int_0^{\infty} e^{-\tau'} \widehat{N}_{\varepsilon}(\phi_{\varepsilon}(\tau'))\,d\tau', \widehat{g}_{\varepsilon} \rangle_{s,k}} \\ \lesssim &\frac{\delta}{C^2}A_{\varepsilon} + \frac{\delta^2}{C^2} M_{\varepsilon} + \delta^2 N_{\varepsilon}
\end{align*}

for again some constants $A_{\varepsilon}, M_{\varepsilon}$ and $N_{\varepsilon} >0$.
If we now choose  $C \geq C_{\varepsilon}$ and $0< \delta < \delta_{\varepsilon} $  with $ C_{\varepsilon}> 1$ sufficiently large and $\delta_{\varepsilon} < 1$ sufficiently small we get that the right-hand side of \eqref{Fixed-point equation for T} is a continuous mapping from $[1-\frac{\delta}{C}, 1+\frac{\delta}{C}]$ into itself so that we obtain by the fixed-point theorem of Brouwer a $T_{\varepsilon} \in \left[ 1- \frac{\delta}{C}, 1 + \frac{\delta}{C}\right]$ such that equation \eqref{Correction term orthogonal to unstable subspace} is fulfilled. We therefore conclude that the corresponding solution $\phi_{\varepsilon}(v,T_{\varepsilon})$ satisfies \eqref{Integral equation with modification}. The claimed uniqueness follows along the lines of the proof of Theorem $5.4$ in \cite{GloKisSch24}.

\end{proof}

Now we will show the regularity of the just constructed solution.

\begin{proposition}\label{Upgrade to classical solution}
Let $v \in \mc{S}_r(\R^n)$ satisfy the assumptions of Theorem \ref{Theorem: Corrected blowup time}. The from Theorem \ref{Theorem: Corrected blowup time} guaranteed solution $\phi_{\varepsilon}$ of Eq.~\eqref{Integral equation with modification} is smooth and satisfies \eqref{Cauchy problem for arbitrary initial data} in the classical sense.
\end{proposition}

\begin{proof}
Due to the regularity of $v$ we first of all obtain $U_{\varepsilon}(v,T_{\varepsilon}) \in \mc{D}(L_{\varepsilon}^X)$ and by the local Lipschitz-continuity of $\widehat{N}_{\varepsilon}$ we obtain from standard semigroup theory that $\phi_{\varepsilon}$ is a strong solution to \eqref{Integral equation with modification}, i.e. it is a strong solution to \eqref{Cauchy problem for arbitrary initial data}.

We therefore have $\phi_{\varepsilon} \in C([0,\infty), \mc{D}(L_{\varepsilon}^X)) \cap C^1([0,\infty), X_s^k(\R^n))$ with
\begin{align}\label{Evolution equation of perturbation in X_s^k}
    \partial_{\tau} \phi_{\varepsilon}(\tau) = L_{\varepsilon} \phi_{\varepsilon}(\tau) + \widehat{N}_{\varepsilon}(\phi_{\varepsilon}(\tau)).
\end{align}

Due to the embedding $X_s^k(\R^n) \hookrightarrow C^2(\R^n)$ \eqref{Evolution equation of perturbation in X_s^k} holds pointwise. If we now use the decomposition $L_{\varepsilon} = L + L_{\varepsilon}'$ and the boundedness of $L_{\varepsilon}'$ we get that $\phi_{\varepsilon}$ solves
\begin{align*}
    \phi_{\varepsilon}(\tau) = S(\tau) U_{\varepsilon}(v,T_{\varepsilon})+\int_0^{\tau}S(\tau-\tau')\left(L_{\varepsilon}'\phi_{\varepsilon}(\tau') + \widehat{N}_{\varepsilon}(\phi_{\varepsilon}(\tau'))\right)\,d\tau'.
\end{align*}

Using the smoothing properties of the free semigroup, see Appendix A.4 from \cite{GloKisSch24}, we obtain $\phi_{\varepsilon}(\tau) \in C^{\infty}(\R^n)$ for every $\tau \geq 0$ and using a generalized version of Schwartz Lemma we are allowed to interchange $\partial_{\tau}$ and $L_{\varepsilon}$ to obtain $\phi_{\varepsilon} \in C^{\infty}([0,\infty) \times \R^n)$.
 
\end{proof}

Now we are finally able to prove our main stability results, Theorem \ref{Theorem: Stability of blowup  solution} and Theorem \ref{Theorem: Stability in normal coordinates}.

\begin{proof}[Proof of Theorem \ref{Theorem: Stability of blowup  solution}]
Under the assumption stated in Theorem \ref{Theorem: Stability of blowup  solution} choose $\omega = \widetilde{\omega}\slash 2$ and $0 < \overline{\varepsilon}$, depending on $\omega$, as at the beginning of Section \ref{Section: Nonlinear Cauchy Problem}. For $\varepsilon \in \R$ with $\abs{\varepsilon} \leq \overline{\varepsilon}$ let  $\delta=\delta_{\varepsilon}$ and $C = C_{\varepsilon}$ denote the constants from Theorem \ref{Theorem: Corrected blowup time} and let $\varphi_0 \in \mc{S}_r(\R^n)$ satisfy $\norm{\varphi_0}_{s,k} < \frac{\delta}{C^2}$.

Then, by Theorem \ref{Theorem: Corrected blowup time} and Proposition \ref{Upgrade to classical solution} there is a  $T = T_{\varepsilon} \in [1-\frac{\delta}{C}, 1 + \frac{\delta}{C}]$ and a unique radial function  $\varphi_{\varepsilon} \in C^{\infty}([0,\infty) \times \R^n)$ solving the initial value problem \eqref{Cauchy problem for perturbation stability analysis}. Moreover,
\begin{align*}
    \norm{\varphi_{\varepsilon}(\tau,\cdot)}_{s,k} \leq \delta e^{-\omega \tau},
\end{align*} 

for all $\tau \geq 0$. We set
\begin{align*}
 v(t,x) : = v_{\varepsilon}^T(t,x) + \frac{1}{\sqrt{T-t}} \varphi_{\varepsilon} \left(\log\left(\frac{T}{T-t}\right) ,\frac{x}{\sqrt{T-t}}   \right). 
\end{align*}

Then $v \in C^{\infty}([0,T)\times \R^n)$ by construction and it satisfies Eq. \eqref{Semilinear heat equation in n dimensions} with
\begin{align*}
    v(0,\cdot) = v_{\varepsilon}^1(0,\cdot) + \varphi_0.
\end{align*} 

Moreover for $r \in [s,k]$,
\begin{align*}
    \norm{\varphi_{\varepsilon}(- \log(T-t) + \log T, \cdot)}_{\dot{H}^r(\R^n)} &  \lesssim \| \phi_{\varepsilon}(\tau) \|_{s,k} \lesssim \delta (T-t)^{\omega} 
\end{align*}

by definition and Theorem \ref{Theorem: Corrected blowup time}.

\end{proof}

\begin{proof}[Proof of Theorem \ref{Theorem: Stability in normal coordinates}]
By the assumptions of Theorem \ref{Theorem: Stability in normal coordinates} the initial data are of the form
\begin{align*}
U_0(x) = U^1_{\varepsilon}(0,x) + x v_0(\abs{x})
\end{align*}

Consequently, $v_0 \in C^{\infty}_e[0,\infty)$ so that $\varphi_0(y) :=   v_0(|y|)$ for $y \in \R^{d+2}$ is a radially symmetric, real-valued Schwartz function $\varphi_0 \in \mc{S}(\R^{d+2})$. By Proposition A.5 and Remark A.6 of \cite{Glo22} there exists a constant $C > 0$ such that 
\begin{align*}
 \norm{\varphi_0}_{X_s^k(\R^{d+2})} \leq C  \norm{\nu_0}_{X_s^k(\R^{d},\R^d)}.
\end{align*}
If the Sobolev exponents $(s,k)$ satisfy condition \eqref{Condition}, then \eqref{Condition on exponents in n-dimensions} holds for $n := d+2$. Let  $\omega, \overline{\varepsilon}, \delta, M >0$ be the constants from Theorem \ref{Theorem: Stability of blowup  solution}. By setting $M_0 := C M$ and requiring 
\begin{align*}
    \norm{\nu_0}_{X_s^k(\R^{d},\R^d)} \leq \frac{\delta}{ M_0},
\end{align*}

we find that $\varphi_0$ satisfies the assumptions of Theorem \ref{Theorem: Stability of blowup  solution}. Hence, there is a $T \in [1-\delta, 1+ \delta]$ and a unique radial solution $v \in C^{\infty}([0,T) \times \R^{d+2})$ to \eqref{Nonlinear heat equation in n dimensions}. If we set $v(t,\cdot) = \tilde v(t,\abs{\cdot})$ then $\tilde v$ solves Eq. \eqref{Cauchy problem 2} for $t \in [0,T)$ and can be written as 
\begin{align*}
    \tilde v(t,\abs{x}) = \frac{1}{\abs{x}} \widetilde{f}_{\varepsilon} \left (\frac{\abs{x}}{\sqrt{T-t}} \right )  + \frac{1}{\sqrt{T-t}} \tilde \varphi \left (\log\left(\frac{T}{T-t}\right), \frac{\abs{x}}{\sqrt{T-t}} \right)
\end{align*}

for $\tilde{\varphi}(t,\cdot) \in  C_e^{\infty}[0,\infty)$ satisfying
\begin{align}\label{Decay for radial phi}
\norm{\tilde{\varphi} (-\log(T-t) + \log T,|\cdot |)}_{\dot{H}^r(\R^{d+2})}\lesssim \delta (T-t)^{\omega}
\end{align}

for all $r \in [s,k]$. We define for $x \in \R^d$, $U(t,x) := x \tilde v(t,\abs{x}) \in C^{\infty}([0,T) \times \R^d, \R^d)$ and find that $U$ can be written as 
\begin{align*}
    U(t,x) = U_{\varepsilon}^T(t,x) + \nu \left (t, \frac{x}{\sqrt{T-t}} \right),
\end{align*}

where $\nu$ is a co-rotational function defined via $\nu(t,x) = x \, \tilde \varphi (-\log(T-t) + \log T, \abs{x})$. The inequality from \eqref{Decay for radial phi} now implies \eqref{Decay of nu} by applying Proposition A.5 and Remark A.6 from \cite{Glo22} and the pointwise as well as the local uniform convergence follow immediately from Sobolev embedding.

\end{proof}

During the preparation of this paper, the author used ChatGPT-5.6 Sol as a supporting tool. It was consulted for language editing and the reformulation of some paragraphs, as well as for some proof strategies in Propositions \ref{Proposition: properties of self-similar solution} and \ref{Projection properties}. All suggestions were independently verified and revised by the author before being incorporated into the manuscript. The author assumes full responsibility for the correctness and integrity of the final work.

\bibliographystyle{plain}
\bibliography{bibliography.bib}

\end{document}